\documentclass[suppldata]{interact}

\usepackage[T2A]{fontenc}
\usepackage[cp866]{inputenc}
\usepackage[english]{babel}
\usepackage{amsmath, amsthm, amssymb, amsfonts}
\usepackage[mathscr]{eucal}
\usepackage{hyperref}

\theoremstyle{plain}
\newtheorem{theorem}{Theorem}
\newtheorem{lemma}[theorem]{Lemma}

\newtheorem{corollary}[theorem]{Corollary}

\theoremstyle{definition}

\newtheorem{remark}[theorem]{Remark}

\DeclareMathOperator{\Rea}{Re}

\DeclareMathOperator*{\esssup}{ess\,sup}
\DeclareMathOperator*{\essinf}{ess\,inf}

\DeclareMathOperator{\Arg}{Arg}

\numberwithin{equation}{section}
\numberwithin{theorem}{section}

\begin{document}
\title{Local and global solutions 
to continuous fragmentation-coagulation equations 
with vanishing diffusion and unbounded 
fragmentation and coagulation rates}%

\author{
\name{S. Shindin}
\affil{School of Agriculture and Science, University of KwaZulu-Natal, Private Bag X54001, 
Durban 4000, South Africa.  E-mail: \href{mailto:shindins@ukzn.ac.za}{\nolinkurl{shindins@ukzn.ac.za}}}
}

\maketitle

\begin{abstract}
In the paper, we study spatially distributed particle systems whose time evolution is governed 
by vanishing diffusion in space $\mathbb{R}^d$, $d\ge 1$, and by size-continuous 
fragmentation and coagulation processes with unbounded rates.
We show that for a large class of coefficients, such systems are classically locally well-posed, 
provided the diffusion and the coagulation processes are suitably dominated by the fragmentation. 
In the special case of power rates, we demonstrate existence of global in time classical solutions 
in all spatial dimensions $d\ge 1$ and without any restrictions on the size of input data.
\end{abstract}


\section{Introduction}\label{sec1}

In the paper, we study the following integro-differential equation
\begin{subequations}\label{eq1.1}
\begin{align}\label{eq1.1a}
u_t &= \mathcal{A} u+ \mathcal{B} u + \mathcal{C}(u,u),\quad u(0) = u_0,\quad t>0,
\end{align}
where $u:=u(x,\xi,t)$, $x\in\mathbb{R}^d$, $d\ge 1$, 
$\xi,t\in\mathbb{R}_+:=(0,+\infty)$ and 
\begin{align}
\label{eq1.1b}
[\mathcal{A}u](x,\xi) &:= \nabla^T \alpha(x,\xi) \nabla u(x,\xi),\\
\nonumber
[\mathcal{B}u](x,\xi) &:=  [\mathcal{B}^+u](x,\xi)- [\mathcal{B}^-u](x,\xi)\\
\label{eq1.1c}
&:=\int_\xi^\infty \gamma(\xi,\eta) \beta(x,\eta) u(x,\eta)d\eta - \beta(x,\xi) u(x,\xi),\\
\nonumber
\mathcal{C}(u,v)(x,\xi) 
&:= \mathcal{C}^+(u,v)(x,\xi) - \mathcal{C}^-(u,v)(x,\xi)\\
\nonumber
&:= \tfrac{1}{4} \int_0^\xi \varkappa(x,\eta,\xi-\eta) 
\bigl[u(x,\eta)v(x,\xi-\eta) + v(x,\eta)u(x,\xi-\eta)\bigr]d\eta \\
\label{eq1.1d}
& - \tfrac{1}{2}\int_{\mathbb{R}_+} \varkappa(x,\eta,\xi) [u(x,\xi)v(x,\eta) + v(x,\xi)u(x,\eta)]d\eta.
\end{align}
\end{subequations}
In a typical scenario, \eqref{eq1.1} describes a spatially distributed particle 
system whose space and size density distribution $u(x,\xi,\cdot)$ is governed 
by the spatial diffusion $\mathcal{A}u$ with rate $\alpha$, by the size breakage/fragmentation 
$\mathcal{B}u$ with rate $\beta$ and by the size coalescence/coagulation $\mathcal{C}(u,u)$
with rate $\varkappa$. In these settings, $L^1$-norms 
\begin{equation}\label{eq1.2}
\|u(t)\|_{L^1(\mathbb{R}^d\times\mathbb{R}_+)},\quad 
\|u(t)\|_{L^1(\mathbb{R}^d\times\mathbb{R}_+,\xi d\xi dx)},\quad t\ge 0,
\end{equation}
represent the total number of particles and the total mass of the evolving system, respectively.

Size-discrete ($\xi\in \mathbb{N}$) and size-continuous ($\xi\in\mathbb{R}_+$)
frag\-mentation-coa\-gulation models, with and without diffusion, 
arise naturally e.g. in the aerosols dynamics \cite{Drake1972, Fr2000},  
in polymers kinetics \cite{AizBak1979, Zif1980, Zif1986},
in astrophysics \cite{Saf1972},
in animal grouping \cite{DegLiuPeg2017},
in phytoplankton dynamics \cite{AckFit1997, AriRud2004, DamDra1995, Jac1990}, etc.,
and hence, attracted a significant attention of the research community. 
In particular a substantial progress is achieved in theoretical studies of space homogeneous 
(diffusion free) discrete and continuous coagulation-fragmentation models. We refer reader 
to the recent monograph \cite{BanLamLau2019I, BanLamLau2019II} for a comprehensive 
discussion of old and new results obtained in this direction. In contrast, mathematical results 
for space inhomogeneous  diffusion-fragmentation-coagulation processes 
are scarce in the literature. 
Most of the available results in this direction are obtained for the discrete breakage/coalescence
models (see e.g. \cite{BenWrz1997,LauMis2002b,Wrz1997,Wrz2004} and discussion in 
\cite[Section]{BanLamLau2019II}), 
while the continuous scenario remains largely unexplored.  
In the latter case, the major technical difficulty is related to the quadratic 
coagulation operator.  In addition to the non-local convolution-like size interactions, 
$\mathcal{C}(u,u)$ involves local in space point-wise products $u(x,\eta,\cdot)u(x,\xi-\eta,\cdot)$ 
and $u(x,\eta,\cdot)u(x,\xi,\cdot)$ that are undefined in physically natural
$L^1\bigr(\mathbb{R}^d\times\mathbb{R}_+, (1+\xi)d\xi dx\bigr)$-settings, see \eqref{eq1.2} above. 
For $d=1$, the problem can be partially alleviated using smoothing and/or hyper-contractivity 
properties of the diffusion process. In higher dimensions however, as well as for unbounded 
fragmentation and coagulation rates, the $L^1$-type spaces are too large to adequately capture 
the dynamics of \eqref{eq1.1}. In these scenarios, suitable choice of a state space is a delicate issue.

The first rigorous study of \eqref{eq1.1} appeared in \cite{Am2000},
where the model is treated as an abstract semilinear parabolic problem in the 
vector-valued Sobolev-Slobodeckii state space 
$W^{1,s}\bigl(\mathbb{R}^d, L^1(\mathbb{R}_+,(1+\xi)d\xi)\bigr)$, $s\notin\mathbb{N}$. 
In this functional setting and for bounded, uniformly elliptic diffusion coefficient and bounded 
coagulation and fragmentation rates, the model is locally-wellposed 
in all space dimensions $d\ge 1$ and yields strong positive mass conserving solutions for 
positive regular data $u_0$. Further, the classical positive solutions are globally defined 
when either $d=1$; or $d\ge 1$ and the diffusion coefficient $\alpha$ is $\xi$-independent. 
We mention that the last assumption can be slightly relaxed using maximal elliptic 
regularity/perturbation results as in e.g. \cite{CanDesFel2014}.

The analysis of \cite{Am2000}, and of the closely related paper \cite{AmWeb2001}, relies on 
a Fourier-multiplier result (see \cite[Theorem~7.3]{Am1997}) that fails for general 
Banach-valued boundary value problems (see the discussion in \cite{Wal2025a}) and/or problems 
with unbounded/vanishing diffusion rates. A potential remedy is to swap roles of the space
($x\in\mathbb{R}^d$) and the size ($\xi\in\mathbb{R}_+$) variables and view solutions as elements 
of  weighted Sobolev-valued state space 
$L^1\bigl(\mathbb{R}_+,(1+\xi)d\xi; W^{s,p}(\mathbb{R}^d)\bigr)$. Though when $p>1$
these state spaces have no obvious physical interpretation, they greatly simplify mathematical analysis of 
the linear diffusion-fragmentation part of the model. In particular, the diffusion semigroup of 
the entire particle system is obtained by gluing together scalar semigroups associated to individual 
particle sizes and hence, there is no need to deal with technically complicated vector-valued settings.  
For models, equipped with bounded, uniformly elliptic diffusion and 
bounded fragmentation and coagulation rates, this approach is adopted in \cite{AmWal2005}.
In that work, global well-posedness  is demonstrated for $d=1$; for $d\ge 1$ and size-independent 
diffusion; and for $d\ge 1$ and general size-dependent diffusion, but in the absence of fragmentation 
and for small initial data. These results were subsequently extended to uniformly elliptic 
unbounded diffusion rates in \cite{Wal2005b}. 

It is worth mentioning that the boundedness assumptions adopted in 
\cite{Am2000,AmWal2005,AmWeb2001} are too restrictive for many realistic applications. 
In general, the diffusion, the fragmentation and the coagulation rates are size-depend and may vary 
by orders of magnitude for very small and very large particles. For instance, it is natural to expect very fast
(unbounded) fragmentation and very slow (vanishing) diffusion rates for extremely large and unstable 
particle aggregates. We note that, the function-theoretic framework of \cite{AmWal2005} is sufficiently 
flexible to accommodate the vanishing/unbounded diffusion and/or the unbounded fragmentation 
and, at least locally, the unbounded coagulation rates. In the absence of coagulation, some results 
in this direction are obtained e.g. in \cite{Ban2007}, while local classical well-posedness of the 
complete nonlinear model \eqref{eq1.1} is investigated recently in \cite{BanMaj2026}. 

To the best of author's knowledge, no global in time well-posedness results for \eqref{eq1.1}
with vanishing diffusion are available in the literature. The only notable exceptions, two 
papers \cite{LauMis2002,MisRod2003}, deal with global weak $L^1$ solutions that need not be unique
and require rather strong structural assumptions (the detailed balance condition) on the coagulation 
and fragmentation kernels. It is the main purpose of this paper to investigated 
global classical solutions to \eqref{eq1.1} with vanishing diffusion and unbounded 
fragmentation and coagulation rates. 

In our analysis, we return to the $L^1\bigl(\mathbb{R}_+, (1+\xi^\ell)d\xi\bigr)$-valued 
settings originated (with $\ell=1$) in \cite{Am2000}. As noted earlier, the vector-valued 
Fourier-multiplier-type generation results of \cite{Am1997} fail for general vanishing diffusion 
coefficients. In the paper, we employ a combination of Green's functions and duality techniques 
to show that suitable realizations of the diffusion ($\mathcal{A}$) and the fragmentation-loss 
($\mathcal{B}^-$) operators generate  positive, analytic, hyper-contractive vector-valued 
$C_0$-semigroup, provided the diffusion is suitably dominated by the fragmentation processes. 
Under additional restrictions on the fragmentation kernel $\gamma$, 
a simple perturbation argument extends the basic generation result from $\mathcal{A}-\mathcal{B}^-$ 
to the complete diffusion-fragmentation operator $\mathcal{A}+\mathcal{B}$. Once this fact is established, 
the standard theory of semi-linear problems yields existence of local in time, mass-conserving, 
positive classical solutions to \eqref{eq1.1}. Existence of global classical solutions requires additional 
structural restrictions on the model coefficients. In the paper, we construct such solutions in the case 
of power diffusion, fragmentation and coagulation rates. To the best of our knowledge this is the first 
result of this type in the literature. 

The paper is organized as follows: in Section~\ref{sec2}, we fix our notation, 
list basic assumption related to the diffusion, fragmentation and coagulation processes 
and outline main results of the paper. The generation of a positive, analytic 
$C_0$-diffusion-fragmentation semigroup is discussed in detail in Sections~\ref{sec3}. 
Section~\ref{sec4} deals with local in time classical solvability of the complete nonlinear 
model \eqref{eq1.1}. Existence of global solutions for \eqref{eq1.1} with power rates 
is investigated in Section~\ref{sec5}. Section~\ref{sec6} concludes the paper.

\section{Assumptions and main results}\label{sec2}

This section is introductory. Here, we fix our notation, list basic assumptions related to 
the diffusion, fragmentation and coagulation operators that appear in \eqref{eq1.1} and 
state main results of the paper. 

\subsection{Notation}\label{sec2.1}
Throughout the paper, we let $\mathbb{R}_+:=(0,\infty)$ and use $\mathcal{S}(d)_+$ to denote 
the cone of positive definite symmetric matrices in $\mathbb{R}^{d\times d}$. For a measurable 
$\alpha : \mathbb{R}^d \to \mathcal{S}(d)_+$, we define
\begin{align*}
&\underline{\lambda}(\alpha) := \essinf_{x\in\mathbb{R}^d}\min_{|z|=1} z^T\alpha (x)z,\quad
|\alpha| :=\esssup_{x\in\mathbb{R}^d }\max_{|z|=1} z^T\alpha(x)z,\quad
\kappa(\alpha) := \tfrac{|\alpha|}{\underline{\lambda}(\alpha)},
\end{align*}
and employ
\[
W_+^{1,\infty}(\mathbb{R}^d; \mathbb{R}^{d\times d}) := 
\bigl\{a\in W^{1,\infty}(\mathbb{R}^d; \mathbb{R}^{d\times d})\,|\,
\alpha : \mathbb{R}^d \to \mathcal{S}(d)_+,\,\underline{\lambda}(\alpha)>0\bigr\},
\]
to denote the cone of positive definite $\mathcal{S}(d)_+$-valued matrix functions in 
$W^{1,\infty}(\mathbb{R}^d; \mathbb{R}^{d\times d})$.
We use $B(x,r)$ and $|B(x,r)|$ for the open ball of radius 
$r>0$ centered at $x\in\mathbb{R}^d$ and for its respective Lebesgue measure. 
For $u\in L^1_{\text{loc}}(\mathbb{R}^d)$, 
\[
\mathcal{M}_\delta[u](x) := \sup_{r>0} |B(x,r)|^{\delta-1} \int_{B(x,r)} |u(y)| dy,\quad \delta\in [0,1),
\]
denotes the standard centered fractional Hardy-Littlewood maximal function.
Throughout the paper, letters $c$ and $C$ are reserved for positive constants whose concrete 
values are irrelevant and may vary from line to line. In general, these constants are controlled
by some parameters. Occasionally, we list them either as subscripts or as arguments. 
We use also $\wedge$ and $\vee$ to denote the standard lattice operations in $\mathbb{R}$. 

For a Banach space $X$, symbols $C_c(\mathbb{R}^d;X)$ and $C_0(\mathbb{R}^d;X)$ 
denote spaces of continuous compactly supported, respectively continuous vanishing at infinity, 
$X$-valued functions. In calculations, we make use also of the weighted space 
\[
C_s\bigl((0,T]; X\bigr) := \bigl\{u\in C((0,T]; X)\,\bigl|\, \lim_{t\to 0^+} t^s\|u(t)\|_X = 0\bigr\},\; 
0<T<\infty,\; s>0.
\]
As mentioned in Section~\ref{sec1}, our analysis makes use of Lebesgue-Bochner spaces 
of $L^1$-valued functions. 
To simplify the notation, for a.e. positive $w\in L^1_{\text{loc}}(\mathbb{R}_+)$ and 
$\underline{\beta}\in L^\infty_{\text{loc},+}(\bar{\mathbb{R}}_+)$, defined below, we let
\begin{align*}
&X^p_{w,s} := L^p(\mathbb{R}^d;L^1\bigl(\mathbb{R}_+,w\underline{\beta}^sd\xi)\bigr),\quad 
s\ge 0,\quad 1\le p<\infty,\\
&X^p_{w} := X^p_{w,0},\quad X^p := X^p_{1,0},\quad 1\le p<\infty.
\end{align*}
In the special case of $w_\ell(\xi):=1+\xi^\ell$, $\ell>0$, we abbreviate
\begin{align*}
&X^p_{\ell,s} := X^p_{w_\ell,s},\quad 
X^p_{\ell} := X^p_{w_\ell,0},\quad \ell,s>0, \quad 1\le p<\infty.
\end{align*}
We let also
\begin{align*}
Y^p_{\ell,s} := X^1_{\ell,s} \cap X^p_{\ell,s},\quad \ell,s\ge 0,\quad 1<p<\infty.
\end{align*}
Finally, throughout the paper we use $\mathcal{X}_+$, 
$\mathcal{X}\in \{X^p_{w,s}, X^p_{\ell,s}, X^p_{w}, X^p_{\ell}, X^p, Y^p_{\ell,s}\}$,
to denote the positive cones of the respective spaces.
 
\subsection{Assumptions}\label{sec2.2}

\subsubsection{The diffusion operator}
The analysis of operator $\mathcal{A}$ relies on the following minimal assumptions:
\begin{subequations}\label{eq2.1}
\begin{align}
\label{eq2.1a}
&\alpha\in L^{\infty}\bigl(\mathbb{R}_+; W_+^{1,\infty}(\mathbb{R}^d; \mathbb{R}^{d\times d})\bigr),\\
\label{eq2.1b}
&\kappa(\alpha) := \esssup_{\xi\in\mathbb{R}_+} \kappa\bigl(\alpha(\xi)\bigr)<\infty.
\end{align} 
\end{subequations}
According to \eqref{eq2.1a}, the diffusion coefficient $\alpha(\xi)$ is upper bounded for 
a.a. values of $\xi\in\mathbb{R}_+$ but is allowed to vanish on a subset of 
$\mathbb{R}_+$ of measure zero and/or as $\xi\to+\infty$. This causes no technical 
difficulties as long as the essential condition number $\kappa(\alpha)$ is globally controlled 
by \eqref{eq2.1b}. 

\subsubsection{The fragmentation operator}
In our treatment of the fragmentation part of \eqref{eq1.1}, we follow closely \cite{Ban2007}.
We assume that
\begin{subequations}\label{eq2.2}
\begin{align}
\label{eq2.2a}
&\beta: \mathbb{R}_+\to L^\infty(\mathbb{R}^d)\;\;
\text{is strongly measurable},\\
\label{eq2.2b}
&0<\beta_0 \le \underline{\beta}(\xi) \le \|\beta(\xi)\|_{L^\infty(\mathbb{R}^d)} 
\le C_\beta \underline{\beta}(\xi)\;\;\text{a.e. in $\mathbb{R}_+$},
\end{align}
for some non-negative 
\begin{equation}\label{eq2.2c}
\underline{\beta}\in L^{\infty}_{\text{loc},+}(\bar{\mathbb{R}}_+)
\end{equation} 
and $C_\beta\ge 1$. We remark that \eqref{eq2.2b}, \eqref{eq2.2c} are natural. In particular, 
the first hypothesis guarantee that the fragmentation process is "isotropic" in space while 
the second one prevents occurrence of "shattering", see e.g. the discussion in 
\cite[Section 5.1.3.3]{BanLamLau2019I}.
\end{subequations}

Throughout the paper, we suppose that the fragmentation kernel $\gamma$ is non-negative, 
measurable in $\mathbb{R}_+^2$ and that the overall fragmentation process is conservative, i.e. 
\begin{subequations}\label{eq2.3}
\begin{align}
\label{eq2.3a}
&\gamma(\xi,\eta)\ge 0,\quad \xi,\eta\in\mathbb{R}_+,\quad  
\gamma(\xi,\eta) = 0,\quad 0<\eta<\xi,\\
\label{eq2.3b}
&\int_0^\eta \xi \gamma(\xi, \eta) d\xi = \eta,\quad \eta\in\mathbb{R}_+.
\end{align}

In connection with the fragmentation kernel $\gamma(\cdot,\cdot)$, we define
\begin{align*}
&\nu_\eta(s) := s \eta \gamma(s \eta,\eta), \quad s\in[0,1],\quad \eta\in\mathbb{R}_+,\\
&\sigma_{0,\ell} := \limsup_{\eta\to\infty} \eta^{-\ell} \int_0^\eta \gamma(\xi, \eta) d\xi 
= \limsup_{\eta\to\infty} \eta^{-\ell} \int_0^1 \nu_\eta(s) \tfrac{ds}{s},\quad \ell\ge 0,\\
&\sigma_{\ell} := \limsup_{\eta\to\infty}
\eta^{-\ell} \int_0^\eta \xi^{\ell} \gamma(\xi, \eta) d\xi
=\limsup_{\eta\to\infty}
\int_0^1 \nu_\eta(s) s^{\ell-1}ds,\quad \ell\ge 0,\\
&\bar{\ell}_0 := \inf\bigl\{ \ell\,\bigr|\, \sigma_{0,\ell'}<\infty,\;0\le \ell\le \ell' \bigr\},\quad
\bar{\ell}_1 := \inf\bigl\{ \ell\,\bigr|\, \sigma_{   \ell'}<\infty,\;0\le \ell\le \ell' \bigr\},
\end{align*}
and assume that 
\begin{align}
\label{eq2.3c}
&0\le \bar{\ell}_0 <\infty,\\
\label{eq2.3d}
&\sigma_{\infty} := \inf_{\ell\ge 1} \sigma_\ell = 0.
\end{align}
\end{subequations}
In addition to \eqref{eq2.2}--\eqref{eq2.3}, our global well-posedness analysis 
relies on the condition
\begin{equation}\label{eq2.4}
0\le \bar{\ell}_0,\bar{\ell}_1 <1.
\end{equation}

Few remarks are in place here. First, assumptions \eqref{eq2.3a} and \eqref{eq2.3b} are natural. 
Second, hypotheses \eqref{eq2.3c} and \eqref{eq2.3d} play an essential role in the analysis 
of the linear part of \eqref{eq1.1}.  To be specific, \eqref{eq2.3c} guarantees that the total number 
of particles appearing in each fragmentation event is globally controlled. 

Third, to attach a physical meaning to \eqref{eq2.3d}, we observe that the closed $s$-interval $[0,1]$ is 
separable and compact. Hence, by Prokhorov's theorem (see e.g. \cite[Section 5]{Bil1999}), 
the family $\mathcal{N}$ of probability measures with densities $\{\nu_\eta(s)\}_{\eta\in\mathbb{R}_+}$ 
is relatively weakly compact in the space of probability measures on $[0,1]$. As under \eqref{eq2.3b}, 
the quantities $\sigma_\ell$ are monotone non-increasing in $\ell$ for $\ell\ge 1$
(see \cite{BanJoSh2019} and  \cite[Theorem 5.1.47(c), pp.~223--224]{BanLamLau2019I}),
after some simple calculations, we arrive at
\[
\nu_\infty(\{1\}) = \sigma_\infty=0,
\]
for all weak accumulation points $\nu_\infty$ of $\mathcal{N}$. In plain words, 
\eqref{eq2.3d} indicates that in the generic situation modeled by $\gamma(\cdot,\cdot)$, 
massive particles almost surely (with the probability of $1-\sigma_\infty=1$) fragment 
into smaller pieces. 

Finally, \eqref{eq2.3d} is stronger than usual assumption that 
$\sigma_{\ell}<1$ for some  $\ell>1$, see e.g. the discussion in 
\cite[Section 5.1.7.2]{BanLamLau2019I} and in particular Remark 5.1.49 on page 226 there.
Likewise, \eqref{eq2.4} is stronger than \eqref{eq2.3c}.  
We note however, that hypotheses \eqref{eq2.3d} and \eqref{eq2.4} are not too restrictive 
and are satisfied for a large class of practical fragmentation kernels, including:
\begin{itemize}
\item[(i)] power kernels 
\[
\gamma(\xi,\eta) = \tfrac{\nu+2}{\eta}\bigl(\tfrac{\xi}{\eta}\bigr)^\nu,\quad -1<\nu\le 0,
\]
where $\bar{\ell}_0=\bar{\ell}_1 = 0$; 
\item[(ii)] homogeneous kernels 
\[
\gamma(\xi,\eta) = \tfrac{1}{\eta}h\bigl(\tfrac{\xi}{\eta}\bigr),\quad h(s)\ge 0,\quad 
\int_0^1 sh(s)ds = 1,
\] 
where $\bar{\ell}_0 = \bar{\ell}_1=0$, provided $h\in L^{p_h}((0,1);s^{\theta_h} ds)$, 
$1<p_h<\infty$ and $0\le \theta_h < 1 - p_h$;
\item[(iii)] separable kernels 
\[
\gamma(\xi,\eta) = \tfrac{\eta h_0(\xi)}{h_1(\eta)},\quad h_0(\xi)\ge 0,\quad  
h_1(\eta) = \int_0^\eta \xi h_0(\xi)d\xi,
\]
where  $\bar{\ell}_0 = 0$, $\bar{\ell}_1 = \tfrac{1}{p_{h_0}}$, 
provided $h_0(\xi)$ satisfies the one-sided doubling and the reverse H\"older inequalities
\begin{align*}
&\int_0^{2\eta} h_0(\xi) d\xi \le c_1\int_\eta^{2\eta} h_0(\xi) d\xi,\\ 
&\Bigr(\tfrac{1}{\eta}\int_0^\eta [\xi h_0(\xi)]^{p_{h_0}} d\xi\Bigl)^{\frac{1}{p_{h_0}}} 
\le \tfrac{c_2}{\eta}\int_0^\eta \xi h_0(\xi)d\xi, 
\end{align*}
with some $1<p_{h_0}<\infty$ and $c_1, c_2>0$, etc.
\end{itemize}

In the presence of vanishing diffusion, conditions \eqref{eq2.1} and \eqref{eq2.2} alone 
are insufficient to guarantee any form of hyper-contractivity of the semigroups 
associated to the linear part of \eqref{eq1.1}. Since the latter is essential 
in the analysis of the quadratic coagulation operator, throughout the paper we assume
\begin{equation}\label{eq2.5}
\kappa_\delta(\alpha,\beta) := 
\essinf_{\xi\in\mathbb{R}_+} |\alpha(\xi)|^\delta \underline{\beta}^{1-\delta}(\xi) >0,
\end{equation}
for some $0\le \delta<1$. 

\subsubsection{The coagulation operator}
In contrast to e.g. \cite{Am2000, AmWal2005, AmWeb2001} and as in e.g. 
\cite{BaJoSh2019b, BanLam2020}, in our analysis we permit unbounded coagulation kernels, 
provided the coagulation is dominated by the fragmentation processes. 
Our approach is close in spirit to \cite{BaJoSh2019b, BanLam2020} and \cite[Section 8.1.2]{BanLamLau2019II} 
where similar ideas are employed in context of simpler discrete and respectively, 
continuous fragmentation--coagulation models. 

To be specific, we assume that the coagulation kernel is measurable in 
$\mathbb{R}_+^2$, non-negative and symmetric, i.e.
\begin{subequations}\label{eq2.6}
\begin{align}
\label{eq2.6a}
&\varkappa:\mathbb{R}^2_+\to L^\infty_+(\mathbb{R}^d)\;\;
\text{is Bochner measurable},\\
\label{eq2.6b}
&\varkappa(\xi,\eta) = \varkappa(\eta,\xi)\;\; \text{a.e. in $\mathbb{R}_+^2$}.
\end{align}
In addition, we assume that the coalescence process is dominated by the fragmentation.
I.e., for some $0<\rho<1$, we have
\begin{equation}\label{eq2.6c}
\|\varkappa(\xi,\eta)\|_{L^\infty(\mathbb{R}^d)} \le c_\varkappa 
\bigl[\underline{\beta}^\rho(\xi) + \underline{\beta}^\rho(\eta)\bigr],
\end{equation}
\end{subequations}
a.e. in $\mathbb{R}_+^2$. 

\subsection{Main results}\label{sec2.3}

In Section~\ref{sec3}, we show that under hypotheses \eqref{eq2.1}--\eqref{eq2.3} and \eqref{eq2.5},  
suitable realizations of the diffusion and the fragmentation operators generate positive, analytic, 
hyper-contractive $C_0$-semigroups in $X^p_\ell$, $1\le p<\infty$, provided the moment 
parameter $\ell$ is sufficiently large.  
To state our first result, for $\ell>\bar{\ell}_0\vee 1$ and $(x,\xi)\in\mathbb{R}^d\times\mathbb{R}_+$, we let
\begin{align*}
&\mathcal{T}_{p,\ell} [u](x,\xi) := \nabla^T \alpha(x,\xi)\nabla u(x,\xi) - \beta(x, \xi) u(x,\xi),\\
&D(\mathcal{T}_{p,\ell}) := \Bigl\{u\in X^p_{\ell},\,\Bigr|\,
\nabla^T\alpha\nabla u,\, \beta u \in X^p_\ell\Bigr\},\\
&\mathcal{B}_{p,\ell}^+[u](x,\xi) := \int_\xi^\infty \gamma(\xi,\eta) 
\beta(x, \eta) u(x, \eta) d\eta,\\
&D(\mathcal{B}_{p,\ell}^+) := \Bigl\{ u\in X^{p}_{\ell}\,\Bigl|\,
\beta u\in X^{p}_{\ell}\Bigr\},
\end{align*}
where the differential expression $\nabla^T\alpha \nabla u$ is understood in the sense of distributions. 
In these settings, we have
\begin{theorem}\label{lm2.1}
Assume that \eqref{eq2.1}, \eqref{eq2.2} and \eqref{eq2.3} are satisfied.
Then for every $1\le p<\infty$, there exists $\ell_p>1$, so that the operator
\[
\bigl(\mathcal{L}_{p,\ell}, D(\mathcal{T}_{p,\ell})\bigr) 
:=\bigl(\mathcal{T}_{p,\ell}+\mathcal{B}_{p,\ell}^+, D(\mathcal{T}_{p,\ell})\bigr)
\]
generates a positive analytic 
$C_0$-semigroup $\bigl\{e^{t\mathcal{L}_{p,\ell}}\bigr\}_{t\ge 0}$
of growth type $\omega_{p,\ell}\ge 0$, in 
$X^p_{\ell}$, for each $\ell>\bar{\ell}_0\vee \ell_p$.
In addition, if \eqref{eq2.5} holds, then
\begin{equation}\label{eq2.7}
\|e^{t\mathcal{L}_{p,\ell}}\|_{
X^p_{\ell}\to 
X^q_{\ell,s}}
\le C_{p,q,s} e^{\omega_{p,\ell} t} t^{-s-\frac{d}{2\delta}(\frac{1}{p}-\frac{1}{q})},
\end{equation}
provided that 
\[
0\le s \le 1 - \tfrac{d}{2\delta}\bigl(\tfrac{1}{p}-\tfrac{1}{q}\bigr),\quad
1\le p\le q< \infty.
\]
The constant  $C_{p,q,s}>1$ is controlled by the quantities $d$, $p$, $q$, $s$, $\ell$,  
$\kappa(\alpha)$ and 
$\kappa_\delta(\alpha,\beta)$ only.
\end{theorem}

By virtue of Theorem~\ref{lm2.1},  assumption \eqref{eq2.6} and hyper-contractivity estimate \eqref{eq2.7}, 
\eqref{eq1.1} is semi-linear in $Y^p_\ell$, $1\le p < \infty$, for $\ell>0$ sufficiently large. 
In Section~\ref{sec4}, we use this fact to prove
\begin{theorem}\label{lm2.2}
Assume that $u_0\in Y^p_{\ell}$, conditions \eqref{eq2.1}--\eqref{eq2.3}, \eqref{eq2.5} and \eqref{eq2.6} 
are satisfied, 
\begin{subequations}\label{eq2.8}
\begin{equation}\label{eq2.8a}
2\le p ,\quad \tfrac{d}{2(1-\rho)\delta}<p <\infty
\end{equation}
and $T=T(u_0)>0$ is sufficiently small. 
\begin{itemize}
\item[(i)] If $\ell>\bar{\ell}_0\vee \ell_{p}$, there exists a unique local mild solution to \eqref{eq1.1} of class 
\begin{align}
\nonumber
u\in Z^p_{\ell,\rho}(T) & := C\bigl([0,T], Y^p_\ell)\cap C_\rho((0,T], Y^p_{\ell,\rho}\bigr)\\
\label{eq2.8b}
&\cap C_{\frac{d}{4p\delta}}\bigl((0,T], X^{2p}_\ell\bigr)
\cap C_{\rho+\frac{d}{4p\delta}}\bigl((0,T], X^{2p}_{\ell,\rho}\bigr).
\end{align}
\item[(ii)] If $\ell>\bar{\ell}_0\vee \ell_{2p}$, the mild solution is classical, i.e. 
\begin{equation}\label{eq2.8c}
u\in C\bigl([0,T], Y^p_\ell\bigr) \cap C^1\bigl((0,T), Y^{p}_\ell\bigr) 
\cap C\bigl((0,T], D(\mathcal{T}_{1,\ell})\cap 
D(\mathcal{T}_{p,\ell})\bigr)
\end{equation}
and furthermore, the flow map $[u_0\mapsto u]: Y^p_\ell\to C\bigl([0,T], Y^p_\ell\bigr)$ is locally 
Lipschitz.
\item[(iii)] Finally, for $\ell>\bar{\ell}_0\vee \ell_{2p}$ and $u_0\in Y^p_{\ell,+}$, 
the classical solution is positive and mass conserving, i.e.
\begin{equation}\label{eq2.8d}
\|u(t)\|_{X^1_{\xi}} = \|u_0\|_{X^1_{\xi}},\quad 0\le t \le T.
\end{equation}
\end{itemize}
\end{subequations}
\end{theorem}

In the presence of vanishing diffusion, global well-posedness analysis requires 
low-order-moment/higher-order-integrability ($0\le \ell\le 1$, $2\le p<\infty$) estimates. 
In Section~\ref{sec5}, we construct such estimates in the special case of power 
diffusion, fragmentation and coagulation rates. To be specific, in addition to \eqref{eq2.1}--\eqref{eq2.6},
we assume
\begin{subequations}\label{eq2.9}
\begin{align}
\label{eq2.9a}
&\alpha(\xi, x) = \alpha(\xi) I_d,\;\;\; 
\underline{c}_\alpha  \le \alpha(\xi)(1+\xi)^{\frac{2}{d}\theta_\alpha} \le \bar{c}_\alpha,
\;\;\; 0<\theta_\alpha,\; 0<\underline{c}_\alpha\le \bar{c}_\alpha,\\
\label{eq2.9b}
&\underline{c}_\beta\le \underline{\beta}(\xi) (1+\xi)^{-\theta_\beta}\le \bar{c}_\beta,
\;\;\; 0<\theta_\beta,\; 0<\underline{c}_\beta\le \bar{c}_\beta,\\
&
\label{eq2.9c}
\text{$\alpha(\cdot)$ and $-\underline{\beta}(\cdot)$ are a.e. non-increasing in $\mathbb{R}_+$},
\end{align}
\end{subequations}
where $I_d\in\mathbb{R}^{d\times d}$ is the identity matrix.
In these settings, we obtain
\begin{theorem}\label{lm2.3}
Assume $\ell>\bar{\ell}_0\vee \ell_{2p}$, $u_0\in Y^p_{\ell,+}$ and conditions 
\eqref{eq2.1}--\eqref{eq2.6}, \eqref{eq2.8a} and \eqref{eq2.9} are satisfied. 
If
\begin{equation}\label{eq2.10}
0<\theta_\alpha< \theta_\beta \wedge (\bar{\ell}_1+(1-\rho)\theta_\beta),\quad
\theta_\beta < \tfrac{2\delta(1-\bar{\ell}_1)p'}{d+2\delta p'} \wedge (1-\bar{\ell}_0),
\end{equation}
then the non-negative classical solutions of Theorem~\ref{lm2.2}(ii)-(iii) are globally defined.
\end{theorem}

Proofs of Theorems~\ref{lm2.1}--\ref{lm2.3} occupy the rest of the paper.

\section{The diffusion-fragmentation process}\label{sec3}

In this section, we show that under assumptions \eqref{eq2.1}--\eqref{eq2.3}, 
a suitable realization of $\mathcal{A}+\mathcal{B}$ from \eqref{eq1.1} generates a 
positive analytic $C_0$-semigroup in $X^p_{\ell}$, $1\le p<\infty$, settings.

\subsection{The scalar diffusion}\label{sec3.1}

To begin, we list standard facts related to the scalar diffusion.
For $a \in W_+^{1,\infty}(\mathbb{R}^d; \mathbb{R}^{d\times d})$
and $b\in L^{\infty}_{+}(\mathbb{R}^d)$, the second-order operator
\begin{equation}\label{eq3.1}
\mathscr{A}_0 u = \nabla^T a \nabla u - b u,\quad u\in \mathcal{D}(\mathbb{R}^d),\\
\end{equation}
is uniformly elliptic. The standard theory of such operators, 
see e.g. \cite[Theorems 1.4.1 and 1.4.2, pp. 22--24]{Dav1989}, implies that
suitable realizations $\bigr(\mathscr{A}_{p}, D(\mathscr{A}_{p})\bigr)$ of 
$\mathscr{A}_{0}$ generate positive semigroups of contractions 
$\bigl\{e^{t\mathscr{A}_{p}}\bigr\}_{t\ge 0}\subset\mathcal{L}\bigl(L^p(\mathbb{R}^d)\bigr)$, 
$1\le p\le \infty$. In particular,
\begin{subequations}\label{eq3.2}
\begin{equation}\label{eq3.2a}
\|e^{t\mathscr{A}_p}\|_{L^p(\mathbb{R}^d)\to L^p(\mathbb{R}^d)} \le e^{-\underline{b}t}, 
\quad \underline{b} := \essinf_{x\in\mathbb{R}^d} b(x),
\quad t\ge 0,\quad 1\le p\le \infty.
\end{equation} 
The semigroups are strongly continuous and analytic when $1\le p<\infty$ 
(see e.g. \cite{Gui1993} for the case of $L^1(\mathbb{R}^d)$, $d\ge 2$). Further,
$\bigr(\mathbb{C}\setminus \bar{\mathbb{R}}_-\bigr) \subset \rho(\mathscr{A}_p) = \rho(\mathscr{A}_q)$, 
$1\le p,q <\infty$, and the semigroups are compatible in the sense that
\begin{align}
\label{eq3.2b}
& e^{z\mathscr{A}_{p}} u = e^{z\mathscr{A}_{q}} u, \;
u \in L^p(\mathbb{R}^d)\cap L^q(\mathbb{R}^d),\; \Rea z>0,\; 1\le p,q\le\infty,\\
\label{eq3.2c}
& \mathscr{A}_{p} u = \mathscr{A}_{q} u, \;
u \in D(\mathscr{A}_{p} )\cap D(\mathscr{A}_{q} ),\; 1\le p,q<\infty.
\end{align}
When $1<p<\infty$, the domains $D(\mathscr{A}_p)$ are independent of 
the diffusion coefficient $a$ (see e.g. \cite[Section 3.1.1, pp.~72--75]{Lun1995})
and are given by 
\begin{equation}\label{eq3.2d}
D(\mathscr{A}_p) = W^{2,p}(\mathbb{R}^d),\quad 1<p<\infty.
\end{equation}
The case of $p=1$ is different, here we have
\begin{equation}\label{eq3.2e}
D(\mathscr{A}_1) := \bigl\{ u\in L^1(\mathbb{R}^d)\,\bigl|\,
\mathscr{A}_0 u \in L^1(\mathbb{R}^d),\,
u\in W^{1,q}(\mathbb{R}^d),\, 1\le q<\tfrac{d}{d-1}
\bigr\},
\end{equation}
where $\mathscr{A}_0 u$ is understood in the sense of distributions, 
see e.g. \cite[Theorem 5.8, p. 213-215]{Tan1996}.

Analyticity of $\bigl\{e^{t\mathscr{A}_{p}}\bigr\}_{t\ge 0}$, together with 
\eqref{eq3.2d}, \eqref{eq3.2e} and the standard Sobolev embedding 
\cite[Theorem 4.12, pp. 85--86]{AdFo2003}, implies that 
\begin{align}
\label{eq3.2f}
&e^{z\mathscr{A}_{1}} \in \mathcal{L}\bigl(L^1(\mathbb{R}^d), L^\infty(\mathbb{R}^d)\bigr),
\quad \Rea z>0.
\end{align}
By the classical result of analysis and by \eqref{eq3.2b}, 
each $e^{z\mathscr{A}_{p}}$, with $1\le p\le \infty$ and $\Rea z>0$ 
is an integral operator, i.e. 
\begin{align}
\label{eq3.2g}
&e^{z\mathscr{A}_{p}}[u](x) = \int_{\mathbb{R}^d} G_z(x,y; a, b) u(y)dy,\\
\label{eq3.2h}
& \bigl\|G_z(\cdot,\cdot; a,b)\bigr\|_{L^\infty(\mathbb{R}^{2d})}\le 
\|e^{z\mathscr{A}_1}\|_{L^1(\mathbb{R}^d)\to L^\infty(\mathbb{R}^d)}.
\end{align}
The kernel $(x,y)\mapsto G_z(x,y; a,b)$ is H\"older continuous, see e.g. \cite[Theorem 5.7, p. 210]{Tan1996} 
for concrete details. Further, $z\mapsto G_z(\cdot,\cdot;a,b)$, 
viewed as an $L^\infty(\mathbb{R}^{2d})$-valued map, is analytic for $\Rea z>0$. 
Finally, in view of the identity 
$\bigl(e^{z\mathscr{A}_p}\bigr)' = e^{\bar{z}\mathscr{A}_{p'}}$, $1< p< \infty$, the kernel
is Hermitian, i.e. 
\begin{equation}\label{eq3.2i}
\overline{G_z(x,y; a,b)} = G_{\bar{z}}(y,x; a,b),\quad x,y\in\mathbb{R}^d, \quad \Rea z>0.
\end{equation}
\end{subequations} 

Green's functions estimates for general, uniformly elliptic operators are very 
well documented in the literature, see e.g. \cite[Theorem 5.7, p. 210]{Tan1996}, 
though constants appearing in these estimates are hard to characterize explicitly 
(e.g. they depend on the best constants in the Agmon-Douglis-Nirenberg a priori estimates). 
In the special case of the self-adjoint, second-order uniformly elliptic operators considered here, 
rather sharp and explicit bounds are available in \cite[Section 3.2, pp. 83--91]{Dav1989}. 
To be specific, the following holds 
\begin{lemma}\label{lm3.1}
Assume $a\in W_+^{1,\infty}(\mathbb{R}^d; \mathbb{R}^{d\times d})$,
$b\in L^\infty_+(\mathbb{R}^d)$ and $\Rea z>0$. Then
\begin{align}\label{eq3.3}
&|G_z(x,y; a,b)| 
\le G_{\Rea z}(|x-y|; a, b):= c_{d} \bigl(\tfrac{\kappa(a)}{|a| \Rea z}\bigr)^{\frac{d}{2}}
e^{-\underline{b} \Rea z - \frac{|x-y|^2}{16 |a| \Rea z}},
\end{align}
where $x,y\in\mathbb{R}^d$ and $c_{d}\ge 1$ is an absolute constant that depends 
on the dimension $d\ge 1$ only. 
Furthermore, $G_t(\cdot,\cdot; a,b)$ is non-negative for $t>0$.
\end{lemma}

Bound \eqref{eq3.3} is standard. However, its explicit dependence on the 
coefficients $a$ and $b$ plays a central role in our analysis of the vector-valued diffusion.
For that reason and for the sake of completeness, we sketch its proof below. 

\begin{proof}
(a) The arguments are essentially those of \cite[Chapters 2--3]{Dav1989}. 
We assume initially that $|a| = 1$ and $b=0$. 
For a given $u_0\in C_0^\infty(\mathbb{R}^d)^+$, $u_0\ne 0$, let $u(t) = e^{t\mathscr{A}_{1}} u_0$, 
$t\ge0$. Since $\|u(t)\|^{2}_{L^{2}(\mathbb{R}^d)}$ is differentiable in $\mathbb{R}_+$,
using the divergence theorem, followed by the Cauchy-Schwarz and Young's inequalities, 
after some simplifications we infer
\begin{align*}
\tfrac{d}{dt}\| u(t) \|_{L^{2}(\mathbb{R}^d)}^{2}
&= 2\bigl\langle u(t), \mathscr{A}_{1} u(t) \bigr\rangle
\le 
-\tfrac{2}{\kappa(a)}  \|\nabla u(t)\|^2_{L^2(\mathbb{R}^d)}.
\end{align*}
Since $u_0\ne 0$, uniqueness of the semigroup solutions 
ensures that neither $\|u(t)\|_{L^{2}(\mathbb{R}^d)}$ nor $\|u(t)\|_{L^{1}(\mathbb{R}^d)}$ 
vanish when $t\ge 0$. Therefore, the classical Nash/Gagliardo-Nirenberg inequality 
\[
\|u\|_{L^2(\mathbb{R}^d)} \le c_{d}' \|\nabla u\|_{L^{2}(\mathbb{R}^d)}^{\frac{d}{d+2}}
\|u\|_{L^1(\mathbb{R}^d)}^{\frac{2}{d+2}},\quad 
u\in W^{1,2}(\mathbb{R}^d)\cap L^1(\mathbb{R}^d),
\]
where $c_{d}'>0$ depends on the dimension $d\ge 1$ only, yields the basic bound
\begin{align*}
\tfrac{4}{c_{d}''\kappa(a)}
\|u(t)\|_{L^{1}(\mathbb{R}^d)}^{-\frac{4}{d}}
\le\tfrac{d}{dt} \| u(t) \|_{L^{2}(\mathbb{R}^d)}^{-\frac{4}{d}},\quad t\ge 0,
\end{align*}
with $c_d'' := d(c_d')^{2+\frac{4}{d}}$.
Since $\bigl\{e^{t\mathscr{A}_1}\bigr\}_{t\ge0}$ are contractions,
upon integration we obtain
\begin{align*}
 \|u(t)\|_{L^{2}(\mathbb{R}^d)}^{-\frac{4}{d}} 
&\ge  
\tfrac{4}{c_{d}''\kappa(a)}
\int_0^t \|u(\tau)\|_{L^{1}(\mathbb{R}^d)}^{-\frac{4}{d}} d\tau
\ge \tfrac{4 t}{c_{d}''\kappa(a)} \|u_0\|_{L^1(\mathbb{R}^d)}^{-\frac{4}{d}},
\quad t\ge 0. 
\end{align*}
Consequently, 
\[
\|u(t)\|_{L^{2}(\mathbb{R}^d)} \le  \bigl(\tfrac{c_d''\kappa(a)}{4 t}\bigr)^{\frac{d}{4}} 
 \|u\|_{L^1(\mathbb{R}^d)},\quad t>0
\]
and, since $C_0^\infty(\mathbb{R}^d)$ is dense in $L^1(\mathbb{R}^d)$, we conclude
\begin{subequations}\label{eqA.1}
\begin{equation}\label{eqA.1a}
\|e^{t\mathscr{A}_1}\|_{L^1(\mathbb{R}^d)\to L^2(\mathbb{R}^d)} \le 
\bigl(\tfrac{c_d''\kappa(a)}{4 t}\bigr)^{\frac{d}{4}},\quad t>0.
\end{equation}
As $e^{t\mathscr{A}_2}$ is symmetric, by the standard duality argument
\begin{equation}\label{eqA.1b}
\|e^{\mathscr{A}_2 t}\|_{L^2(\mathbb{R}^d)\to L^\infty(\mathbb{R}^d)} \le 
\bigl(\tfrac{c_d''\kappa(a)}{4 t}\bigr)^{\frac{d}{4}},\quad t>0.
\end{equation}
\end{subequations}
In connection with \eqref{eqA.1}, we note that the best value of $c_d'$ is available 
e.g. in \cite{CarLos1993}.

(b) We fix $\varphi\in C_b^\infty(\mathbb{R}^d)$, with $\|\nabla\varphi\|_{L^\infty(\mathbb{R}^d)}\le 1$.
In view of \eqref{eqA.1}, we may repeat verbatim the arguments of 
Theorem~2.2.3, pp. 64--66, 
Lemma~3.2.1, pp. 83--84, 
Corollary~3.2.2, pp. 84--85 and 
Theorem~3.2.4, pp. 87 from \cite{Dav1989} (in that order), to obtain
\begin{equation}\label{eqA.2}
\|e^{\mu\varphi}e^{t\mathscr{A}_{1} }e^{-\mu\varphi}\|_{L^1(\mathbb{R}^d)\to L^\infty(\mathbb{R}^d)} \le 
\bigl(\tfrac{c_d''\kappa(a)}{2 t}\bigr)^{\frac{d}{2}} e^{2\mu^2 t},\quad t>0,\quad \mu\in\mathbb{R}.
\end{equation} 
As the space $\mathcal{L}\bigl(L^1(\mathbb{R}^d), L^\infty(\mathbb{R}^d)\bigr)$ is in one-to-one 
correspondence with the space of integral operators of the form \eqref{eq3.2g} satisfying \eqref{eq3.2h}, 
it follows from \eqref{eqA.1} and \eqref{eqA.2} that 
\begin{align*}
0\le |G_t(x,y;a,0)| \le \bigl(\tfrac{c_d''\kappa(a)}{2t} \bigr)^{\frac{d}{2}}
e^{2\mu^2 t  + \mu (\varphi(y) - \varphi(x))},\quad t>0,\quad \mu\in\mathbb{R}. 
\end{align*}
Since $\mu\in\mathbb{R}^d$ and $\varphi\in C_b^\infty(\mathbb{R}^d)$ were arbitrary, 
we let $\mu = \tfrac{(\varphi(x) - \varphi(y))}{4t}$ and 
$\varphi(z) = \tfrac{(x-y)\cdot z}{|x-y|}$, for $|z|\le 2(|x|\vee|y)|$, to obtain
\begin{subequations}\label{eqA.3}
\begin{equation}\label{eqA.3a}
0\le |G_t(x,y; a,0)|\le  \bigl(\tfrac{c_d''\kappa(a)}{2t} \bigr)^{\frac{d}{2}} 
e^{-\frac{|x-y|^2}{8t}},\quad t>0,\quad x,y\in\mathbb{R}^d,
\end{equation}
in the special case of $|a| = 1$ and $b=0$. The general bound
\begin{equation}\label{eqA.3b}
0\le |G_t(x,y; a,b)|\le  \bigl(\tfrac{c_d''\kappa(a)}{2|a|t} \bigr)^{\frac{d}{2}} 
e^{-\frac{|x-y|^2}{8|a|t}},\quad t>0,\quad x,y\in\mathbb{R}^d,
\end{equation}
\end{subequations}
follows from \eqref{eqA.3a}:
(i) by observing that $G_t(x,y; a,b) \le e^{ - \underline{b}t}G_t(x,y; a,0)$, $t\ge 0$;
and (ii) by substituting $\tau = |a| t$. 
Finally, when $\Rea z>0$, the direct application of \cite[Theorem 3.4.8, pp. 103--104]{Dav1989} 
to \eqref{eqA.3b} yields \eqref{eq3.3}.
The proof is complete.
\end{proof}

We remark that \eqref{eq3.3}, together with the convolution Young
inequality, yields the bound
\begin{equation}\label{eq3.4}
 \|e^{z\mathscr{A}_p}\|_{L^p(\mathbb{R}^d)\to L^q(\mathbb{R}^d)}
\le c_{d} 
\tfrac{\kappa(a)^{\frac{d}{2}} e^{-\underline{b} \Rea z}}{(|a|\Rea z)^{\frac{d}{2}(\frac{1}{p}-\frac{1}{q})}},
\quad 1\le p\le q\le \infty,\quad \Rea z >0,
\end{equation}
where as before $c_d\ge 1$ is an absolute constant that depends on the dimension $d\ge 1$ only.

In addition to Lemma~\ref{lm3.1}, we need the following technical 
\begin{lemma}\label{lm3.2}
For $a\in W_+^{1,\infty}(\mathbb{R}^d; \mathbb{R}^{d\times d})$, 
$b\in L^\infty_+(\mathbb{R}^d)$ and $\Rea z>0$:
\begin{itemize}
\item[(i)] the map $(a,b) \mapsto e^{z\mathscr{A}_p}$ is continuous in the uniform operator topology 
from $W^{1,\infty}_+(\mathbb{R}^d; \mathbb{R}^{d\times d})\times L^\infty_+(\mathbb{R}^d)$ to 
$\mathcal{L}\bigl(L^p(\mathbb{R}^{d})\bigr)$, $1<p<\infty$;
\item[(ii)] the map $(a,b) \mapsto e^{z\mathscr{A}_p}$ is continuous in the strong operator topology 
from $W^{1,\infty}_+(\mathbb{R}^d; \mathbb{R}^{d\times d})\times L^\infty_+(\mathbb{R}^d)$ to 
$\mathcal{L}\bigl(L^p(\mathbb{R}^{d}), L^q(\mathbb{R}^{d})\bigr)$, $1\le p\le q <\infty$.
\end{itemize}
\end{lemma}
\begin{proof}
(a) Let, in addition to $(a,b)$, another pair 
$(\alpha, \beta)\in  W^{1,\infty}_+(\mathbb{R}^d; \mathbb{R}^{d\times d})\times  
L^\infty_+(\mathbb{R}^d)$ of coefficients 
be fixed and let $\bigl(\mathscr{B}_p, D(\mathscr{B}_p)\bigr)$ be the realization of 
$\mathscr{B}_0 := \nabla^T \alpha\nabla - \beta I$ that generates 
the diffusion semigroup 
$\bigl\{e^{\mathscr{B}_p z}\bigr\}_{\{\Rea z>0\}\cup\{0\}}\subset \mathcal{L}\bigl(L^p(\mathbb{R}^d)\bigr)$,
$1\le p\le \infty$.
 
For $u\in W^{2,p}(\mathbb{R}^d)$, $1<p<\infty$, the domain identity \eqref{eq3.2d} 
and the classical  Agmon-Douglis-Nirenberg 
a priori estimate, see e.g. \cite[Section 4.2, pp. 130--131]{Tan1996}, yields the basic bound
\begin{align*}
\bigl\|[\mathscr{A}_p - \mathscr{B}_p] u\bigr\|_{L^p(\mathbb{R}^d)} 
&\le c\bigl[\|a-\alpha\|_{W^{1,\infty}(\mathbb{R}^d;\mathbb{R}^{d\times d})}
+\|b-\beta\|_{L^{\infty}(\mathbb{R}^d)}\bigr]
\|u\|_{W^{2,p}(\mathbb{R}^d)}\\
&\le c \bigl[\|a-\alpha\|_{W^{1,\infty}(\mathbb{R}^d;\mathbb{R}^{d\times d})}
+\|b-\beta\|_{L^{\infty}(\mathbb{R}^d)}\bigr]
\|\mathscr{A}_pu\|_{L^p(\mathbb{R}^d)},
\end{align*}
with some $c >0$ that depends on $p$, $d$, 
$\|a\|_{W^{1,\infty}(\mathbb{R}^d;\mathbb{R}^{d\times d})}$,  
$\|b\|_{L^{\infty}(\mathbb{R}^d)}$ only.
Hence, for $\lambda\in \rho(\mathscr{A}_p)\cap \rho(\mathscr{B}_p)$, 
\begin{align*}
&\bigl\|(\lambda I - \mathscr{A}_p)^{-1} - (\lambda I  - \mathscr{B}_p)^{-1}\bigr\|_{L^p(\mathbb{R}^d)\to 
L^p(\mathbb{R}^d)} \\
&\qquad\qquad\qquad\qquad
\le \tfrac{c'}{|\lambda|}
\bigl[\|a-\alpha\|_{W^{1,\infty}(\mathbb{R}^d;\mathbb{R}^{d\times d})}
+\|b-\beta\|_{L^{\infty}(\mathbb{R}^d)}\bigr],
\end{align*}
with $c' >0$ controlled by $p$, $d$, $\|a\|_{W^{1,\infty}(\mathbb{R}^d;\mathbb{R}^{d\times d})}$,
$\|\alpha\|_{W^{1,\infty}(\mathbb{R}^d;\mathbb{R}^{d\times d})}$ and 
$\|b\|_{L^{\infty}(\mathbb{R}^d)}$,
$\|\beta\|_{L^{\infty}(\mathbb{R}^d)}$. 
Since $\bigl(\mathscr{A}_p, W^{2,p}(\mathbb{R}^d)\bigr)$ and 
$\bigl(\mathscr{B}_p, W^{2,p}(\mathbb{R}^d)\bigr)$ are 
both sectorial, we can estimate the difference $e^{z\mathscr{A}_p} - e^{z\mathscr{B}_p}$ using 
the Dunford integral representations for the semigroups 
$\bigl\{e^{z\mathscr{A}_p}\bigr\}_{\{\Rea z>0\}\cup\{0\}}$
and $\bigl\{e^{z\mathscr{B}_p}\bigr\}_{\{\Rea z>0\}\cup\{0\}}$. 
Arguing as in e.g. \cite[Proposition 4.3, pp. 97--100]{EnNa2000}, we obtain
\begin{equation}\label{eq3.5}
\|e^{z\mathscr{A}_p} - e^{z\mathscr{B}_p}\|_{L^p(\mathbb{R}^d)\to L^p(\mathbb{R}^d)} 
\le c'' \bigl[\|a-\alpha\|_{W^{1,\infty}(\mathbb{R}^d;\mathbb{R}^{d\times d})}
+\|b-\beta\|_{L^{\infty}(\mathbb{R}^d)}\bigr],
\end{equation}
with some $c''>0$, controlled by the same quantities as $c'$. Hence, the first claim is settled.

(b) Assertion (ii) follows from \eqref{eq3.4} and \eqref{eq3.5}, by writing 
\[
\bigl(e^{z\mathscr{A}_p} - e^{z\mathscr{B}_p}) u = 
\bigl(e^{z\mathscr{A}_p} - e^{z\mathscr{B}_p}) [I-e^{(tz)\mathscr{A}_p}]u 
+ \bigl(e^{z\mathscr{A}_q} - e^{z\mathscr{B}_q})e^{(tz)\mathscr{A}_p}u,
\]
for a suitably chosen $0<t<1$.
\end{proof}

\subsection{The loss-diffusion semigroup}\label{sec3.2} 

We turn now to the vector-valued scenario. For the collection of kernels 
\[
\bigl\{G_z\bigl(\cdot,\cdot;\alpha(\xi), \beta(\xi)\bigr)\bigr\}_{\xi, \Rea z>0},
\]
from Lemma~\ref{lm3.1} and for 
$u\in \mathcal{D}(\mathbb{R}^d\times\mathbb{R}_+)$, we define \emph{formally}
\begin{subequations}\label{eq3.6}
\begin{align}
\label{eq3.6a}
& e^{0\mathcal{T}_0}[u] := u,\\
\label{eq3.6b}
&e^{z\mathcal{T}_0}[u](x,\xi) := 
\int_{\mathbb{R}^d} G_z\bigl(x,y; \alpha(\xi), \beta(\xi)\bigr) u(y,\xi) dy,\\
\label{eq3.6c}
&\mathcal{T}_0[u](x,\xi) := \mathcal{T}_0(\xi)[u(\cdot,\xi)](x) = 
\nabla^T \alpha(x,\xi)\nabla u(x,\xi) - \beta(x, \xi) u(x,\xi).
\end{align}
\end{subequations}
The following is the main result of this Section
\begin{theorem}\label{lm3.3}
Let an a.e. positive weight $w\in L^1_{\text{loc}}(\mathbb{R}_+)$ be fixed. 
Under assumptions \eqref{eq2.1} and \eqref{eq2.2}:
\begin{subequations}\label{eq3.7}
\begin{itemize}
\item[(i)]
The family $\bigl\{e^{z\mathcal{T}_0}\bigr\}_{\{\Rea z>0\}\cup\{0\}}$ of the integral 
maps defined in \eqref{eq3.6} extends to a positive analytic $C_0$-semigroup:
\begin{itemize}
\item[(i.a)] of contractions in $X^1_w$, i.e.
\begin{equation}\label{eq3.7a}
\|e^{t\mathcal{T}_1}\|_{X^1_w
\to X^1_w} \le 1,\quad t\ge 0;
\end{equation}
\item[(i.b)] of growth type zero in  $X^p_w$, 
$1\le p<\infty$, i.e.
\begin{equation}\label{eq3.7b}
\|e^{z\mathcal{T}_p}\|_{X^p_w \to X^p_w} \le C_p,\quad z\in\{\Rea z>0\}\cup\{0\},
\end{equation}
for some $C_p\ge 1$ that depends on $d$, $\kappa(a)$ and $p$ only.
\end{itemize}
\item[(ii)] For each $1\le p<\infty$ and $\Rea z>0$, the semigroups 
actions are given by the same integral formula \eqref{eq3.6b}.
\item[(iii)] If in addition \eqref{eq2.5} holds, the semigroups are hyper-contractive, i.e.
for $1\le p\le q < \infty$, $s\ge 0$ and $\Rea z>0$, we have 
\begin{equation}\label{eq3.7c}
\|e^{z\mathcal{T}_p}\|_{X^p_w \to X^q_{w,s}} \le C_{p,q,s} (\Rea z)^{-s-\frac{d}{2\delta}(\frac{1}{p}-\frac{1}{q})},
\end{equation}
where $C_{p,q,s}>1$ is controlled by $d$, $p$, $q$, $s$ and the quantities 
$\kappa(\alpha)$ and $\kappa_\delta(\alpha,\beta)$ only.
Further, for $\lambda>0$, the resolvents of the generators satisfy
\begin{equation}\label{eq3.7f}
(\lambda-\mathcal{T}_p)^{-1}\in \mathcal{L}\bigl(X^p_w, X^q_{w,s}\bigr),
\end{equation}
provided
\[
0\le s\le 1-\tfrac{d}{2\delta}\bigl(\tfrac{1}{p}-\tfrac{1}{q}\bigr),\quad 1\le p\le q< \infty.
\]
\item[(iv)] The semigroups generators $\bigl(\mathcal{T}_p, D(\mathcal{T}_p)\bigr)$, $1\le p<\infty$, 
are given explicitly by
\begin{align}
\label{eq3.7d}
&\mathcal{T}_p u := \mathcal{T}_0u,\\
\label{eq3.7e}
&D(\mathcal{T}_p) := \Bigl\{u\in X^p_{w},\,\Bigr|\,
\mathcal{T}_0 u, \beta u \in X^p_w\Bigr\},
\end{align}
where \eqref{eq3.7d} is understood in the sense of distributions. 
\end{itemize}
\end{subequations}
\end{theorem}

The proof of Theorem~\ref{lm3.3} is straightforward but lengthy. For the readers convenience, 
we prove each assertion separately. We start with the case of $p=1$. In these settings, 
\[
X^1_w = 
L^1\bigl(\mathbb{R}_+, wd\xi; L^1(\mathbb{R}^d)\bigr) 
= L^1(\mathbb{R}^d\times \mathbb{R}_+ , wd\xi dx),
\]
by the Fubini theorem, which significantly simplifies calculations.

\begin{proof}[Proof of Theorem~\ref{lm3.3}(i.a) and (ii) with $p=1$.] $\,$\\
(a) Let some $z$, with $\Rea z>0$, and $u\in X^1_{w}$ be fixed. 
We define 
\[
\mathfrak{A} := \bigl\{a\in W^{1,\infty}_+(\mathbb{R}^d;\mathbb{R}^{d\times d})\,\bigl|\, \kappa(a)\le 
\kappa(\alpha)\bigr\}.
\]
By \eqref{eq2.1}, \eqref{eq2.2} and the Fubini theorem, the vector valued map
\[
\mathbb{R}_+\ni \xi\mapsto \bigl(\alpha(\xi),\beta(\xi), u(\xi)\bigr) \in 
\mathfrak{A}\times L_+^\infty(\mathbb{R}^d)\times L^1(\mathbb{R}^d)
\] 
is strongly measurable and hence can be realized as the (a.e. in $\mathbb{R}_+$) point-wise limit 
of a sequence of simple functions 
\[
\xi\mapsto \bigl(\alpha_n(\xi),\beta_n(\xi), u_n(\xi)\bigr),\quad n\ge 1.
\] 
As set $\mathfrak{A}$ is closed in $W^{1,\infty}(\mathbb{R}^d;\mathbb{R}^{d\times d})$, 
this sequence can be chosen so as to take its non-zero values in 
$\mathfrak{A}\times L_+^\infty(\mathbb{R}^d)\times L^1(\mathbb{R}^d)$, 
this is easy to verify e.g. by repeating verbatim the proof of the Pettis 
measurability theorem, see e.g. \cite[Theorem 1.1.1, pp. 7--8]{ArBaHiNe2011}. 
On the account of Lemmas~\ref{lm3.1} and \ref{lm3.2}(b), the map 
\[
\mathfrak{A}\times L_+^\infty(\mathbb{R}^d)\times L^1(\mathbb{R}^d)\ni (a,b,u) 
\mapsto  
 \int_{\mathbb{R}^d} G_z(\cdot,y; a,b) u(y)dy \in L^1(\mathbb{R}^d)
\] 
is continuous. Using this fact, it is not difficult to verify that the $L^1(\mathbb{R}^d)$-valued 
sequence of simple functions
\[
\xi\mapsto
\left\{
\begin{array}{ll}
\int_{\mathbb{R}^d} G_z\bigl(\cdot,y; \alpha_n(\xi),\beta_n(\xi)\bigr) u_n(y,\xi)dy,& \alpha_n(\xi) \ne 0,\\
e^{-\beta_n(\xi) z} u_n(\xi),& \alpha_n(\xi) = 0,
\end{array}
\right.
\] 
is well defined and converges a.e. in $\mathbb{R}_+$ to 
\[
[S_{\alpha,\beta}(z)u](\xi) := e^{z\mathcal{T}_1(\xi)} u(\xi)
= \int_{\mathbb{R}^d} G_z(\cdot,y; \alpha(\xi),\beta(\xi)) u(y,\xi)dy.
\]
We conclude that the 
$L^1(\mathbb{R}^d)$-valued map $\xi\mapsto [S_{\alpha,\beta}(z)u](\xi)$ 
is strongly measurable in $\mathbb{R}_+$ and, subsequently, $S_{\alpha,\beta}(z)$, 
viewed as an operator-valued map from $\mathbb{R}_+$ to $\mathcal{L}\bigl(L^1\bigl(\mathbb{R}^d)\bigr)$,
is strongly measurable in the strong operator topology.

(b) In view of \eqref{eq3.3} and \eqref{eq3.2a}, we have 
\begin{align*}
&\bigl\|S_{\alpha,\beta}(z) u \bigr\|_{X^1_w} 
\le C_1\|u\|_{X^1_w},\quad \Rea z>0,\\
&\|S_{\alpha,\beta}(t) u \|_{X^1_w} 
\le \|u\|_{X^1_w},\quad t\ge 0.
\end{align*}
By the definition of $S_{\alpha,\beta}(z)$, 
\[
\bigl\|[S_{\alpha,\beta}(z)u](\xi) - u(\xi)\bigr\|_{L^1(\mathbb{R}^d)}\to 0\quad\text{ as $z\to 0$, $\Rea z>0$},
\] 
for a.a. $\xi \in \mathbb{R}_+$.
Hence, $S_{\alpha,\beta}(z)u\to u$ as $z\to 0$, $\Rea z>0$, 
in $X^1_w$ by the dominated convergence theorem. 
Since the identity 
$S_{\alpha,\beta}(z')S_{\alpha,\beta}(z'') u = S_{\alpha,\beta}(z'+z'')u$, $\Rea z',\Rea z''>0$, 
is obvious, we conclude that 
\[
\bigl\{e^{z\mathcal{T}_1}\bigr\}_{\{\Rea z>0\}\cup\{0\}}:=
\bigl\{S_{\alpha,\beta}(z)\bigr\}_{\{\Rea z>0\}\cup\{0\}}
\] 
is a $C_0$-semigroup of contractions in $X^1_w$. 
The positivity of $\bigl\{e^{t\mathcal{T}_1}\bigr\}_{t\ge0}$, follows from the positivity of 
$\bigl\{e^{t\mathcal{T}_1(\xi)}\bigr\}_{t\ge 0}$ for a.a. $\xi\in\mathbb{R}_+$.

(c) The preceding calculations and \eqref{eq3.3} indicate that 
$\bigl\{e^{te^{i\theta_0}\mathcal{T}_1}\bigr\}_{t\ge 0}$ is a $C_0$-semigroup of 
growth type zero in $\mathcal{L}\bigl(X^1_w\bigr)$
for every fixed value $0\le \theta_0<\tfrac{\pi}{2}$.
Consequently, choosing $r>0$ and some $0< \theta_0<\tfrac{\pi}{2}$, on the account 
of \eqref{eq3.6b}, \eqref{eq3.3} and the Fubini theorem, we obtain
\begin{align*}
\bigl\|(ir-\mathcal{T}_1)^{-1}u\bigr\|_{X^1_w}
&= \bigl\|e^{-i\theta_0} (re^{i(\frac{\pi}{2}-\theta_0)} - e^{-i\theta_0}\mathcal{T}_1)^{-1}
u\bigr\|_{X^1_w} \\
&
\le \int_{\mathbb{R}_+} e^{-rt\sin\theta_0} 
\|e^{te^{-i\theta_0}\mathcal{T}_1}u\|_{X^1_w}dt
\le \tfrac{c}{r}\|u\|_{X^1_w},
\end{align*}
where $c\ge 1$ depends on $d$, $\kappa(a)$ and $\theta_0$ only. 
The last bound is equivalent to the analyticity of  
$\bigl\{e^{z\mathcal{T}_1}\bigr\}_{\{\Rea z>0\}\cup\{0\}}$.

(d) To conclude the proof, we note that by the construction of 
$e^{z\mathcal{T}_1}$ and by the Fubini theorem, the 
$L^1(\mathbb{R}_+, wd\xi)$-valued map $x\to [e^{z\mathcal{T}_1}u](x)$ is Bochner integrable 
and agrees a.e. in $\mathbb{R}^d$ with the formal integral expression \eqref{eq3.6b} 
for every fixed $u\in X^1_w$ and $\Rea z>0$. The proof is complete.
\end{proof}

Proofs of assertions (i.b), (ii) and (iii) rely on the following technical observation:
\begin{lemma}\label{lm3.4}
Assume \eqref{eq2.1} and \eqref{eq2.2} are satisfied. If $u\in L^1_{\text{loc}}(\mathbb{R}^d)$, then
\begin{subequations}\label{eq3.8}
\begin{equation}\label{eq3.8a}
e^{\underline{\beta}(\xi)\Rea z}\bigl|e^{z\mathcal{T}_1(\xi)}[u](x)\bigr| 
\le c_{0,0} \mathcal{M}_{0}[u](x),\; z\in \{\Rea z>0\}\cup\{0\},
\end{equation}
for a.a. $\xi\in\mathbb{R}_+$ and $x\in\mathbb{R}^d$.
In addition, if \eqref{eq2.5} holds, then also 
\begin{align}
\label{eq3.8b}
&\bigl|\underline{\beta}(\xi)^s e^{z\mathcal{T}_1(\xi)}[u](x)\bigr| 
\le c_{s,\delta'} (\Rea z)^{-s-\frac{d\delta'}{2\delta}} \mathcal{M}_{\delta'}[u](x),\; 0\le \delta'<1, 
\; s\ge 0,\; \Rea z>0,\\
\label{eq3.8c}
&\bigl|\underline{\beta}(\xi)^s (\lambda-\mathcal{T}_1(\xi))^{-1}[u](x)\bigr| 
\le c_{s} \mathcal{M}_{\frac{(1-s)2\delta}{d}}[u](x),\; 0\le s\le 1,\; 1-\tfrac{d}{2\delta}<s,\; \lambda>0,
\end{align}
\end{subequations}
for a.a. $\xi\in\mathbb{R}_+$ and $x\in\mathbb{R}^d$.
The constants $c_{s,\delta'}, c_{s}>0$ are controlled by 
$d$, $\delta$, $\delta'$, $s$ and the quantities 
$\kappa(\alpha)$ and $\kappa_\delta(\alpha,\beta)$ only. 
\end{lemma}
\begin{proof} We prove \eqref{eq3.8b} and \eqref{eq3.8c} only, the proof of \eqref{eq3.8a} is almost 
identical to that of \eqref{eq3.8b}.

(a) Let $z$, with $\Rea z>0$, be fixed. 
By \eqref{eq3.3}, $\bigl|G_z\bigl(x,y; \alpha(\xi), \beta(\xi)\bigr)\bigr|$ 
is dominated by the radially decreasing kernel $G_{\Rea z}\bigl(|x-y|; \alpha(\xi),\beta(\xi)\bigr)$
for a.a. $\xi \in\mathbb{R}_+$.  
Hence, for any $\zeta>0$, and a.e. $\xi\in\mathbb{R}_+$, we have 
\begin{align*}
&\bigl\{G_{\Rea z}\bigl(|y|; \alpha(\xi),\beta(\xi)\bigr)>\zeta\bigr\} 
= B\bigl(0,r_\zeta(\xi)\bigr) :=\bigl\{y\,\bigr|\, |y|<r_\zeta(\xi)\bigr\},\\
&\bigl\{G_{\Rea z}\bigl(|y|^{\frac{1}{\theta}}; \alpha(\xi),\beta(\xi)\bigr)>\zeta\bigr\} 
= B\bigl(0,r_\zeta(\xi)^\theta\bigr),\quad \theta>0,
\end{align*} 
for some $r_\zeta(\xi)>0$.
It turns out that if $u\in L^1_{\text{loc}}(\mathbb{R}^d)$, then
\begin{align*}
&\underline{\beta}(\xi)^s\int_{\mathbb{R}^d} \bigl|G_z\bigl(x,y; \alpha(\xi), \beta(\xi)\bigr) u(y)\bigr|dy
\le \underline{\beta}(\xi)^s \int_{\mathbb{R}_+} d\zeta \int_{B(0,r_\zeta(\xi))} |u|(x-y) dy\\
&\qquad
\le \underline{\beta}(\xi)^s \int_{\mathbb{R}_+} |B(0,r_\zeta(\xi))|^{1-\delta'} d\zeta
\mathcal{M}_{\delta'}[u](x)\\
&\qquad
= \underline{\beta}(\xi)^s \int_{\mathbb{R}_+} |B(0,r_\zeta(\xi)^{1-\delta'})| d\zeta
\mathcal{M}_{\delta'}[u](x)\\
&\qquad
= \underline{\beta}(\xi)^s 
\bigl\|G_{\Rea z}\bigl(|\cdot|^\frac{1}{1-\delta'};\alpha(\xi),\beta(\xi)\bigr)\bigr\|_{L^1(\mathbb{R}^d)}
\mathcal{M}_{\delta'}[u](x)\\
&\qquad
= c\kappa(\alpha)^{\frac{d}{2}}
(|\alpha(\xi)| \Rea z)^{-\frac{d}{2}\delta'}  \underline{\beta}(\xi)^s  e^{-\underline{\beta}(\xi) \Rea z}
\mathcal{M}_{\delta'}[u](x),
\end{align*}
for a.a. $x\in \mathbb{R}^d$.
Since in view of \eqref{eq2.5},
\begin{align*}
&(|\alpha(\xi)|\Rea z)^{-\frac{d}{2}\delta'}  \underline{\beta}(\xi)^s  e^{-\underline{\beta}(\xi) \Rea z}\\
&\qquad\qquad
\le \bigl[c \kappa_\delta(\alpha,\beta)^{-\frac{d\delta'}{2\delta}}
\sup_{\zeta>0} \zeta^{s+\frac{d\delta'(1-\delta')}{2\delta}} e^{-\zeta}\bigr]
(\Rea z)^{-s-\frac{d\delta'}{2\delta}}
= c (\Rea z)^{-s-\frac{d\delta'}{2\delta}},
\end{align*}
for a.a. $\xi\in \mathbb{R}_+$, we arrive at \eqref{eq3.8b}.

(b) In view of \eqref{eq3.3}, for $\lambda>0$, we have 
\begin{align*}
&\bigr|(\lambda - \mathcal{T}_1(\xi))^{-1}[u](x)\bigl| \le 
\int_{\mathbb{R}_+} e^{-\lambda t} dt \int_{\mathbb{R}^d} 
G_t\bigl(x-y; \alpha(\xi),\beta(\xi)\bigr) |u(y)|dy\\
&\qquad\qquad\qquad\qquad
= \int_{\mathbb{R}^d} K_\lambda\bigl(x-y; \alpha(\xi),\beta(\xi)\bigr) |u(y)| dy,\\
&K_\lambda(x;\alpha(\xi),\beta(\xi)) \le  c_d \bigl(\tfrac{\kappa(\alpha(\xi))}{|\alpha(\xi)|}\bigr)^{\frac{d}{2}}
\bigl(\lambda+\underline{\beta}(\xi)\bigr)^{\frac{d}{2}-1} 
K_d\left(|x|\bigl(\tfrac{\lambda+\underline{\beta}(\xi)}{4|\alpha(\xi)|}\bigr)^\frac{1}{2}\right),\\
&K_d(r) := \int_{\mathbb{R}^+} t^{-\frac{d}{2}} e^{-t-\frac{r^2}{4t}} dt,\quad r> 0,
\end{align*}
where $c_d>1$ is an absolute constant that depends on the dimension $d\ge 1$ only.
It is well known (see e.g. \cite[Section 1.6]{Hen1993}) that
\[
K_d(r) \le \bar{K}_d(r) :=c e^{-r} \left\{
\begin{array}{ll}
r^{2-d}, &d>2, \\
\bigl|\ln r \bigr|, &d = 2,\\
1, & d=1,  
\end{array}
\right.
\quad r >0.
\]
We note that $\bar{K}_d(r)$ is monotone decreasing if $d\ne 2$ and that
there exists a continuous monotone decreasing fiunctions $\hat{K}_2(r)$, that satisfies 
\begin{align*}
&\bar{K}_2(r) = \hat{K}_2(r),\quad r\in \mathbb{R}_+\setminus \bigl[\tfrac{1}{2},2\bigr];\quad
\bar{K}_2(r) \le \hat{K}_2(r),\quad r\in\mathbb{R}_+.
\end{align*}
For $d\ne 2$, we let $\hat{K}_d(r) := \bar{K}_d(r)$ and proceed as in part (a) of the 
proof to obtain
\begin{align*}
&\bigl|\underline{\beta}(\xi)^s\bigl(\lambda - \mathcal{T}_1(\xi)\bigr)^{-1} [u](x)\bigr| \\
&\qquad\qquad
\le c \bigl(\tfrac{\lambda+\underline{\beta}(\xi)}{4|\alpha(\xi)|}\bigr)^{\frac{d}{2} - (1-s)\delta}
\int_{\mathbb{R}^d} 
\hat{K}_d\left(|x|\bigl(\tfrac{\lambda+\underline{\beta}(\xi)}{4|\alpha(\xi)|}\bigr)^\frac{1}{2}\right) |u(x-y)|dy\\
&\qquad\qquad
\le c \int_{\mathbb{R}+} \left|\bigl\{\hat{K}_d(|\cdot|)>\zeta\bigr\}\right|^{1-\frac{(1-s)2\delta}{d}} d\zeta 
\mathcal{M}_{\frac{(1-s)2\delta}{d}}[u](x)\\
&\qquad\qquad
\le c\bigl\|\hat{K}_d(|\cdot|^\frac{d}{d-(1-s)2\delta})\bigr\|_{L^1(\mathbb{R}^d)}
\mathcal{M}_{\frac{(1-s)2\delta}{d}}[u](x)
 = c \mathcal{M}_{\frac{(1-s)2\delta}{d}}[u](x),
\end{align*}
for a.a. $x\in\mathbb{R}^d$ and $\xi\in\mathbb{R}_+$.
\end{proof}

\begin{proof}[Proof of Theorem~\ref{lm3.3}(i.b) and (ii) with $1<p<\infty$.] $\,$\\
(a) Let $1< p<\infty$ be fixed. Assume initially $u\in C_c\bigl(\mathbb{R}^d; L^1(\mathbb{R}_+, wd\xi)\bigr)$.
Then by Theorem~\ref{lm3.3}(i.a),  
the $L^1(\mathbb{R}_+, wd\xi)$-valued map $e^{z\mathcal{T}_1}u$ is Bochner integrable. 
Consequently, when $\Rea z>0$ and $v\in C_c(\mathbb{R}^d)$, the quantity 
\[
\bigl\langle \|e^{z\mathcal{T}_1} u\|_{L^1(\mathbb{R}_+,wd\xi)}, |v| \bigr\rangle
\] 
is finite. On the account of the Fubini theorem and Lemma~\ref{lm3.4},
we have 
\begin{align*}
&\bigl|\bigl\langle \|e^{z\mathcal{T}_1} u\|_{L^1(\mathbb{R}_+,wd\xi)}, v \bigr\rangle \bigr|\\
&\qquad\qquad
\le \int _{\mathbb{R}^d} \int_{\mathbb{R}_+} |u(\xi,x)| 
\int_{\mathbb{R}^d} \bigl|G_{\bar{z}}\bigl(x, y;\alpha(\xi),\beta(\xi)\bigr)\bigr| |v(y)| dy wd\xi dx\\
&\qquad\qquad
\le c_{0,0}\int_{\mathbb{R}^d} \|u(x)\|_{L^1(\mathbb{R}_+,wd\xi)} \mathcal{M}_{0}[v] (x) dx\\
&\qquad\qquad
\le c_{0,0} \|u\|_{X^p_w} \bigl\|\mathcal{M}_{0}[v] \bigr\|_{L^{p'}(\mathbb{R}^d)}
\le C_{p} \|u\|_{X^p_w} \|v\|_{L^{p'}(\mathbb{R}^d)},
\end{align*}
where $C_{p}:=c_{0,0}\|\mathcal{M}_0\|_{L^{p'}(\mathbb{R}^d)\to L^{p'}(\mathbb{R}^d)}>0$
is finite, provided $1\le p<\infty$. 
Since the class $C_c(\mathbb{R}^d)$ is dense in $L^{p'}(\mathbb{R}^d)$, we conclude 
that
\[
\|e^{z\mathcal{T}_1} u\|_{X^p_w} \le C_{p} \|u\|_{X^p_w},
\]
for data in $C_c\bigl(\mathbb{R}^d; L^1(\mathbb{R}_+, wd\xi)\bigr)$.
Since the latter class is dense in $X^p_w$, 
it follows that each $e^{z\mathcal{T}_1}$, $\Rea z>0$, extends to a bounded 
linear map $e^{z\mathcal{T}_p}$ in $X^p_w$ and that \eqref{eq3.7b} is satisfied.

(b) Verification of the strong continuity, the semigroup identity and the analyticity for data in 
$C_c\bigl(\mathbb{R}^d; L^1(\mathbb{R}_+, wd\xi)\bigr)$ is straightforward while the general result 
follows by density. The positivity of $\bigl\{e^{t\mathcal{T}_p}\bigr\}_{t\ge0}$ 
follows from the positivity of Green's functions $G_t\bigl(\cdot,\cdot,\alpha(\xi),\beta(\xi)\bigr)$, $t>0$. 

(c) To complete the proof, we note that $L^{p'}\bigl(\mathbb{R}^d; C_0(\mathbb{R}_+)\bigr)$
is a closed subspace of $(X^{p})^\ast$ that is norming for 
$X^{p}$, $1< p<\infty$. 
Accordingly, for $v\in L^{p'}\bigl(\mathbb{R}^d; C_0(\mathbb{R}_+)\bigr)$, 
$u\in C_c\bigl(\mathbb{R}^d; L^1(\mathbb{R}_+, wd\xi)\bigr)$, $1< p<\infty$ and $\Rea z>0$, 
the Fubini theorem and Theorem~\ref{lm3.3}(i.a) and (ii), with $p=1$, give
\[
\langle v, we^{z\mathcal{T}_p} u  \rangle 
= \langle v, w e^{z\mathcal{T}_1} u  \rangle 
= \langle e^{\bar{z}\mathcal{T}_1} v, w u  \rangle.
\] 
By density, this identity extends to 
\[
\langle v, we^{z\mathcal{T}_p} u  \rangle 
= \langle e^{\bar{z}\mathcal{T}_1} v, w u  \rangle,
\] 
for all $v\in L^{p'}\bigl(\mathbb{R}^d; C_0(\mathbb{R}_+)\bigr)$ and 
$u\in X^p_w$. In particular, 
using the Fubini theorem one more time, we obtain
\[
e^{z\mathcal{T}_p}[u](x) = e^{z\mathcal{T}_1}[u](x) = 
\int_{\mathbb{R}^d} G_z\bigl(x,y; \alpha(\cdot), \beta(\cdot)\bigr) u(y) dy,
\]
in $X^p_w$. The proof is complete.
\end{proof}

Bound \eqref{eq3.7b} indicates that in the case of $1<p<\infty$, the loss-diffusion semigroups 
are not necessarily contractive. Moreover, it is well known \cite[Section~1.5, p. 8]{St1970} that 
\[
\|\mathcal{M}_{0}\|_{L^{p'}(\mathbb{R}^d)\to L^{p'}(\mathbb{R}^d)} = \mathcal{O}(p),\quad p\to\infty,
\]
and our construction of $e^{z\mathcal{T}_p}$ fails for $p=\infty$.

As noted in Section~\ref{sec2}, conditions \eqref{eq2.1} and \eqref{eq2.2} alone 
are insufficient to guarantee any form of hyper-contractivity of the semigroups 
$\bigl\{e^{z\mathcal{T}_p}\bigr\}_{\{\Rea z>0\}\cup\{0\}}$, $1\le p<\infty$. If however 
\eqref{eq2.5} holds, the vector valued loss-diffusion semigroups of 
Theorem~\ref{lm3.3} behave similar to the fractional scalar diffusion. 

\begin{proof}[Proof of Theorem~\ref{lm3.3}(iii).] $\,$\\
(a) Bound \eqref{eq3.7c} follows from \eqref{eq3.8b} by duality. Indeed, for $\Rea z>0$,
$v\in L^{q'}(\mathbb{R}^d)$ and $u\in C_c\bigl(\mathbb{R}^d; L^1(\mathbb{R}_+, wd\xi)\bigr)$, 
the Fubini theorem, combined with H\"older's inequality, yields
\begin{align*}
&\bigl|\bigl\langle \|e^{z\mathcal{T}_p} u\|_{L^1(\mathbb{R}_+,\underline{\beta}^s wd\xi)}, v\bigr\rangle\bigr|
\le c_{s,\frac{1}{q'}-\frac{1}{p'}}(\Rea z)^{-s-\frac{d}{2\delta}(\frac{1}{q'}-\frac{1}{p'})}
\bigl\|\mathcal{M}_{\frac{1}{q'}-\frac{1}{p'}}[v]\bigr\|_{L^{p'}(\mathbb{R}^d)} 
\|u\|_{X^p_w}\\
&\qquad\qquad\qquad\qquad
\le C_{p,q,s} (\Rea z)^{-s-\frac{d}{2\delta}(\frac{1}{q'}-\frac{1}{p'})}\|v\|_{L^{q'}(\mathbb{R}^d)} 
\|u\|_{X^p_w},
\end{align*}
with 
\[
C_{p,q,s}:=c_{s,\frac{1}{q'}-\frac{1}{p'}} 
\bigl\|\mathcal{M}_{\frac{1}{q'}-\frac{1}{p'}}\bigr\|_{L^{q'}(\mathbb{R}^d)\to L^{p'}(\mathbb{R}^d)}.
\]
The latter quantity is known to be finite, provided $1\le p\le q < \infty$.
For $1<p< q <\infty$ see e.g. \cite[Theorem 3]{MuckWed1974},
the case $1=p< q<\infty$ follows from H\"olders inequality, while the case $1\le p=q<\infty$
is classical. Hence, \eqref{eq3.7c} follows by density.

(b) The inclusion \eqref{eq3.7f}, with $0\le s<1-\tfrac{d}{2\delta}\bigl(\tfrac{1}{p}-\tfrac{1}{q}\bigr)$, 
follows immediately from \eqref{eq3.7c} and the integral identity 
$(\lambda-\mathcal{T}_p)^{-1}u = \int_{\mathbb{R}_+} e^{-t(\lambda-\mathcal{T}_p)}udt$.
Further, proofs of Theorem~\ref{lm3.3}(i) and (ii), together with the same integral identity, 
indicate that 
\begin{subequations}\label{eq3.9}
\begin{align}
\label{eq3.9a}
&(\lambda - \mathcal{T}_p)^{-1}[u](x) 
= \int_{\mathbb{R}^d} K_\lambda\bigl(x,y; \alpha(\cdot), \beta(\cdot)\bigr) u(y) dy,\\
\label{eq3.9b}
& K_\lambda\bigl(x,y; \alpha(\cdot), \beta(\cdot)\bigr) = \int_{\mathbb{R}_+} e^{-\lambda t} 
G_t\bigl(x,y; \alpha(\cdot), \beta(\cdot)\bigr)dt,\quad \lambda>0,
\end{align}
\end{subequations}
for $u\in X^p_w$, $1\le p <\infty$ and a.e. 
$\xi\in\mathbb{R}_+$. Using this fact, bound \eqref{eq3.8c} and the duality argument similar to 
that employed in part (a) of the proof, we obtain \eqref{eq3.7f} with 
$s=1-\tfrac{d}{2\delta}\bigl(\tfrac{1}{p}-\tfrac{1}{q}\bigr)$ and $1\le p\le q<\infty$.
\end{proof}

It remains to characterize the generators $\bigl(\mathcal{T}_p, D(\mathcal{T}_p)\bigr)$, $1\le p<\infty$.  
\begin{proof}[Proof of Theorem~\ref{lm3.3}(iv)] $\,$\\
(a) Let $u\in C_c\bigr(\mathbb{R}^d; L^1(\mathbb{R}_+, wd\xi)\bigl)$ and $\lambda>0$ be fixed. 
We denote $f = (\lambda - \mathcal{T}_1)^{-1}u$. In view of Theorem~\ref{lm3.3}(i), (ii) and 
\eqref{eq3.9}, we have $f\in \bigcap_{p\ge 1} X^p_w$. 
Further, when $v\in \mathcal{D}(\mathbb{R}^d\times \mathbb{R}_+)$, under assumptions 
\eqref{eq2.1} and \eqref{eq2.2}, we have 
\[
\mathcal{T}_0v \in L^\infty\bigl(\mathbb{R}_+, L^\infty(\mathbb{R}^d)\bigr)\subsetneq 
L^\infty(\mathbb{R}^d\times\mathbb{R}_+)
\]
and therefore, the point-wise product $(\mathcal{T}_0v)wf$ is absolutely integrable. 
Using this fact, the Fubini theorem, \eqref{eq3.2e}, \eqref{eq3.9} and partial integration, we infer
\begin{align*}
\bigl\langle (\lambda- \mathcal{T}_0)v, wf\bigr\rangle
& 
 = \int_{\mathbb{R}_+}w(\xi)d\xi \int_{\mathbb{R}^d} \bigl(\lambda - \mathcal{T}_0(\xi)\bigr) [v](\xi, x) 
 \bigl(\lambda - \mathcal{T}_1(\xi)\bigr)^{-1} [u](x,\xi) dx\\
& = \int_{\mathbb{R}_+}w(\xi)d\xi \int_{\mathbb{R}^d} v(x,\xi) u(x,\xi) dx
= \langle v, wu\rangle.
\end{align*}
Since by density, this identity extends to 
$u\in X^p_w$, $1\le p<\infty$, we conclude that 
$\mathcal{T}_0 f$ is a regular Schwartz distribution for any $f\in D(\mathcal{T}_p)$, $1\le p<\infty$.

Further, since the Schwartz class $\mathcal{D}(\mathbb{R}^d\times \mathbb{R}_+)$
is dense in $C_0(\mathbb{R}^d\times\mathbb{R}_+)$ and in 
$L^{p'}\bigl(\mathbb{R}^d; C_0(\mathbb{R}_+)\bigr)$, $1<p<\infty$, and 
since  the latter spaces are norming for $X^1$ and for $X^{p}$, 
$1< p<\infty$, respectively, it follows also that 
\[
\bigl\langle v, w (\lambda- \mathcal{T}_0)f\bigr\rangle = \langle v, wu\rangle,
\]
for all $v\in C_0(\mathbb{R}^d\times\mathbb{R}_+)$, 
$u\in X^{1}_w$ and 
$v\in L^{p'}\bigl(\mathbb{R}^d; C_0(\mathbb{R}_+)\bigr)$, 
$u\in X^{p}_w$, $1<p<\infty$.
In other words, we have
\[
\mathcal{T}_0:D(\mathcal{T}_p)\to X^{p}_w, \quad
\bigl(\lambda - \mathcal{T}_0\bigr) \bigl(\lambda - \mathcal{T}_p\bigr)^{-1} = I,
\]
for each $1\le p<\infty$ and $\lambda>0$, provided 
all partial derivatives are understood in the sense of distributions.

Conversely, for $u\in X^{p}_w$, 
with $\mathcal{T}_0u\in X^{p}_w$,
$\lambda>0$ and 
$v\in\mathcal{D}(\mathbb{R}^d\times\mathbb{R}_+)$, the Fubini theorem, \eqref{eq3.2d}, 
\eqref{eq3.2e} and partial integration give 
\[
\bigl\langle v, w\bigl(\lambda-\mathcal{T}_p\bigr)^{-1}\bigl(\lambda-\mathcal{T}_0\bigr)u\bigr\rangle
=\bigl\langle \bigl(\lambda-\mathcal{T}_0\bigr)\bigl(\lambda-\mathcal{T}_p\bigr)^{-1}v, wu\bigr\rangle
=\langle v, wu\rangle.
\]
Hence, every such $u$ is an element of the domain $D(\mathcal{T}_p)$.

(b) For $u\in D(\mathcal{T}_p)$, the inclusion $\beta u\in X^{p}_w$ follows from \eqref{eq3.7f}, 
with $1\le p=q<\infty$ and $s=1$. The proof is complete.
\end{proof}

\subsection{The diffusion-fragmentation semigroup}\label{sec3.3}

Now we are in the position to prove Theorem~\ref{lm2.1}.
For $1\le p<\infty$ and $\ell>\bar{\ell}_0\vee 1$, we let (see Subsection~\ref{sec2.3})
\begin{subequations}\label{eq3.10}
\begin{align}
\label{eq3.10a}
&\mathcal{T}_{p,\ell} [u](x,\xi) := \nabla^T \alpha(x,\xi)\nabla u(x,\xi) - \beta(x, \xi) u(x,\xi),\\
\label{eq3.10b}
&D(\mathcal{T}_{p,\ell}) := \Bigl\{u\in X^p_{\ell},\,\Bigr|\,
\nabla^T\alpha\nabla u,\, \beta u \in X^p_\ell\Bigr\},\\
\label{eq3.10c}
&\mathcal{B}_{p,\ell}^+[u](x,\xi) := \int_\xi^\infty \gamma(\xi,\eta) 
\beta(x, \eta) u(x, \eta) d\eta,\\
\label{eq3.10d}
&D(\mathcal{B}_{p,\ell}^+) := \Bigl\{ u\in X^{p}_{\ell}\,\Bigl|\,
\beta u\in X^{p}_{\ell}\Bigr\},\\
\label{eq3.10e}
&\bigl(\mathcal{L}_{p,\ell}, D(\mathcal{T}_{p,\ell})\bigr) 
:=\bigl(\mathcal{T}_{p,\ell}+\mathcal{B}_{p,\ell}^+, D(\mathcal{T}_{p,\ell})\bigr).
\end{align}
\end{subequations}
\begin{proof}[Proof of Theorem~\ref{lm2.1}]
Our treatment of the vector-valued diffusion-frag\-men\-ta\-tion operator
$\mathcal{T}_{p,\ell}+\mathcal{B}_{p,\ell}^+$, $1\le p<\infty$, follows closely the technique 
employed in \cite[Theorem 5.1.48, pp. 224-225]{BanLamLau2019I} in context 
of the scalar fragmentation.

(a) We begin with the case of $1<p<\infty$.
For $u\in X^p_{\ell}$, with $\ell>\bar{\ell}_0\vee 1$ to be chosen later, 
we define $u_n(x,\xi) = u(x,\xi)\chi_{[0,n]}(\xi)$, $n\ge 1$. 
By Theorem~\ref{lm3.3}(iv), $(\lambda - \mathcal{T}_{p,\ell})^{-1} u_n \in D(\mathcal{B}_{p,\ell}^+)$, 
for any $\ell\ge 0$, $\lambda>0$ and $n\ge 1$. Hence, if 
$v\in L^{p'}(\mathbb{R}^d)$, $n\ge 1$, the quantities
\begin{align*}
&\bigl|\bigl\langle \|\mathcal{B}_{p,\ell}^+ (\lambda -\mathcal{T}_{p,\ell})^{-1}
u_n\|_{L^1(\mathbb{R}_+,w_\ell d\xi)}, 
v \bigr\rangle\bigr|\\
&\qquad\qquad\qquad
 = \Bigl|\Bigl\langle \Bigl\|\mathcal{B}_{p,\ell}^+ \int_{\mathbb{R}_+} 
e^{-t(\lambda-\mathcal{T}_{p,\ell})}u_n dt\Bigr\|_{L^1(\mathbb{R}_+,w_\ell d\xi)}, 
v \Bigr\rangle\Bigr|
\end{align*}
are finite. Using this fact together with the Fubini theorem and Lemma~\ref{lm3.4}, as in the proof of 
Theorem~\ref{lm3.3}(i.b) we infer
\begin{align*}
&\bigl|\bigl\langle \|\mathcal{B}_{p,\ell}^+ (\lambda -\mathcal{T}_{p,\ell})^{-1}
u_n\|_{L^1(\mathbb{R}_+,w_\ell d\xi)}, v \bigr\rangle\bigr|
\le c_{d,\alpha,\beta} \int_{\mathbb{R}^d} \mathcal{M}_{0}[v](x)dx\\
&\qquad\qquad\qquad\qquad\qquad\qquad
\int_{\mathbb{R}_+} 
\Bigl(\int_0^\eta \gamma(\xi,\eta) w_\ell (\xi)d\xi\Bigr) 
\tfrac{\underline{\beta}(\eta)}{\lambda+\underline{\beta}(\eta)} |u_n|(x,\eta) d\eta,
 \end{align*}
where $c_{d,\alpha,\beta}\ge 1$ depends on $d$, $\kappa(a)$ and $C_\beta$ from \eqref{eq2.2b} 
only. To estimate the inner integral with respect to $\eta$, for a given $\ell>\bar{\ell}_0$, we choose 
$\bar{\eta}>0$ so large that $\eta^{-\ell}\int_0^1 \nu_\eta(s)s^{\ell-1}\le 2\sigma_\ell$, 
$\eta\ge \bar{\eta}$.
Then in view of \eqref{eq2.3b}, \eqref{eq2.3c} and the inequality $\ell>\bar{\ell}_0$, a.e. in $\mathbb{R}^d$ 
we have 
\begin{align*}
&\int_{\mathbb{R}_+} 
\Bigl(\int_0^\eta \gamma(\xi,\eta) w_\ell (\xi)d\xi\Bigr) 
\tfrac{\underline{\beta}(\eta)}{\lambda+\underline{\beta}(\eta)} |u_n|(x,\eta) dx\\
&\qquad\qquad
\le (1+\sigma_{0,\bar{\ell}_0}) \int_0^{\bar{\eta}} \tfrac{\underline{\beta}(\eta)}{\lambda+\underline{\beta}(\eta)}
|u_n|(x,\eta) w_\ell(\eta) d\eta\\
&\qquad\qquad
+ \Bigl(\sigma_{0,\bar{\ell}_0}\tfrac{1+\bar{\eta}^{\bar{\ell}_0}}{1+\bar{\eta}^\ell} + 2\sigma_\ell\Bigr) 
\int_{\bar{\eta}}^\infty 
\tfrac{\underline{\beta}(\eta)}{\lambda+\underline{\beta}(\eta)}
|u_n|(x,\eta) w_\ell(\eta) d\eta\\
&\qquad\qquad
\le \Bigl(\tfrac{1+\sigma_{0,\bar{\ell}_0}}{\lambda}
\|\underline{\beta}\chi_{[0,\bar{\eta}]}\|_{L^\infty(\mathbb{R}_+)}
+ \sigma_{0,\bar{\ell}_0}\tfrac{1+\bar{\eta}^{\bar{\ell}_0}}{1+\bar{\eta}^\ell} + 2\sigma_\ell\Bigr)
\|u_n(x)\|_{L^1(\mathbb{R}_+,w_\ell d\xi)}.
\end{align*}   
We denote 
$C_{d,p,\alpha,\beta}: = c_{d,\alpha,\beta}\|\mathcal{M}_0\|_{L^{p'}(\mathbb{R}^d)\to L^{p'}(\mathbb{R}^d)}$.
In view of \eqref{eq2.3d}, there exists $\ell_p>1$ so that 
$\sigma_\ell \le \tfrac{1}{8C_{d,p,\alpha,\beta}}$ for all $\ell> \ell_p$. Hence, for every fixed 
$\ell>\bar{\ell}_0\vee \ell_p$, we can choose  $\bar{\eta}>0$ so large that 
$\sigma_{0,\bar{\ell}_0}\tfrac{1+\bar{\eta}^{\bar{\ell}_0}}{1+\bar{\eta}^\ell}\le \tfrac{1}{4C_{d,p,\alpha,\beta}}$
and then $\lambda>0$ so large that $\tfrac{1+\sigma_{0,\bar{\ell}_0}}{\lambda}
\|\underline{\beta}\chi_{[0,\bar{\eta}]}\|_{L^\infty(\mathbb{R}_+)}\le \tfrac{1}{4C_{d,p,\alpha,\beta}}$.
In these settings, we obtain
\begin{align*}
&\bigl|\bigl\langle \|\mathcal{B}_{p,\ell}^+ (\lambda -\mathcal{T}_{p,\ell})^{-1}
u_n\|_{L^1(\mathbb{R}_+,w_\ell d\xi)}, 
v \bigr\rangle\bigr|\\
&\qquad\qquad\qquad
\le \tfrac{3}{4} \|u\|_{X^p_{\ell}} \|v\|_{L^{p'}(\mathbb{R}^d)},
\quad n\ge 1.
\end{align*}
Since $u\in X^p_{\ell}$, 
$v\in L^{p'}(\mathbb{R}^d)$ were arbitrary and all operators are positive,  we employ 
the Lebesgue monotone convergence theorem to conclude that
\begin{equation}\label{eq3.12}
\|\mathcal{B}_{p,\ell}^+ (\lambda -\mathcal{T}_{p,\ell})^{-1}\|_{
X^p_{\ell}\to X^p_{\ell}} \le \tfrac{3}{4},
\quad \ell>\bar{\ell}_0\vee \ell_p.
\end{equation}

(b) Bound \eqref{eq3.12} indicates that
\[
(\lambda - \mathcal{T}_{p,\ell} - \mathcal{B}_{p,\ell}^+)^{-1} = 
(\lambda - \mathcal{T}_{p,\ell} )^{-1}\sum_{n\ge 0} \bigl[\mathcal{B}_{p,\ell}^+
(\lambda - \mathcal{T}_{p,\ell} )^{-1}\bigr]^n,
\] 
in $X^p_{\ell}$, provided $\ell>\bar{\ell}_0\vee \ell_p$ and 
$\lambda>0$ is sufficiently large. 
Since $(\lambda - \mathcal{T}_{p,\ell})^{-1}$ (see Theorem~\ref{lm3.3}(i)) and $\mathcal{B}_{p,\ell}^+$ 
are positive, the resolvent $(\lambda - \mathcal{T}_{p,\ell} - \mathcal{B}_{p,\ell}^+)^{-1}$ is positive in 
$X^p_{\ell}$. Since the latter space is a Banach lattice
and since the semigroup $\{e^{z\mathcal{T}_{p,\ell}}\}_{\{\Rea z>0\}\cup\{0\}}$ is positive and analytic, 
the perturbation theorem of W. Arendt and A. Rhandi, \cite[Theorem 1.1]{ArRh1991}, 
settles the claim for $1<p<\infty$. 
Calculations in $X^1_{\ell,s}$ settings are almost identical 
and are omitted. The proof of the first assertion is complete.

(c) To obtain \eqref{eq2.7}, we note that on the account of 
\cite[Theorems~5.1.2, p. 107 and 5.5.3, p.120]{BeLo1976}, we have
\begin{align}
\label{eq3.13}
&\bigl[X^{p_0}_{\ell}, X^{p_1}_{\ell,s}\bigr]_\theta
= X^{p}_{\ell,\theta s},\\
\nonumber
&0<\theta<1,\quad\tfrac{1}{p} = \tfrac{1-\theta}{p_0}+\tfrac{\theta}{p_1},\quad
1\le p_0\le p_1\le\infty,\quad p_0\ne \infty,
\end{align}
with equal norms, where $[\cdot,\cdot]_\theta$ is the standard complex interpolation functor.
On the other hand, by part (a) of the proof 
\[
\bigl(\lambda - \mathcal{T}_{p,\ell}\bigr)\bigl(\lambda-\mathcal{L}_{p,\ell}\bigr)^{-1}
\in\mathcal{L}\bigl(X^{p}_{\ell}\bigr), \quad 1\le p<\infty,
\]
for $\lambda>\omega_{p,\ell}\ge 0$, where $\omega_{p,\ell}$ is the growth type 
of $\bigl\{e^{t\mathcal{L}_{p,\ell}}\bigr\}_{t\ge 0}$.
Hence, \eqref{eq2.7} follows from the analyticity of 
$\bigl\{e^{t\mathcal{L}_{p,\ell}}\bigr\}_{t\ge 0}$, inclusion \eqref{eq3.7f} and the standard 
interpolation argument, by writing
\[
e^{-t(\lambda - \mathcal{L}_{p,\ell})} = 
\bigl(\lambda - \mathcal{T}_{p,\ell}\bigr)^{-1}
\Bigl[\bigl(\lambda - \mathcal{T}_{p,\ell}\bigr)\bigl(\lambda-\mathcal{L}_{p,\ell}\bigr)^{-1}\Bigr]
\Bigl[\bigl(\lambda-\mathcal{L}_{p,\ell}\bigr) e^{-t(\lambda - \mathcal{L}_{p,\ell})}\Bigr]
\]
and then estimating the product term by term.
\end{proof}

To conclude this section, we note that the perturbation argument, employed in Theorem~\ref{lm2.1}, 
requires 
\[
\sigma_{\ell_p} = \mathcal{O}\bigl(\| \mathcal{M} \|^{-1}_{ L^{p'}(\mathbb{R}^d)\to L^{p'}(\mathbb{R}^d)}\bigr)
 = \mathcal{O}(p^{-1}),\quad p\to \infty,
\]
while for the power, homogeneous and separable kernels, mentioned in Subsection~\ref{sec2.2}, 
we have $\sigma_\ell = \mathcal{O}(\ell^{-1})$, $\sigma_\ell \le  c\ell^{\frac{1}{p_h}-1}$
and $\sigma_\ell \le c\ell^{\frac{1}{p_{h_0}}-1}$, respectively. 
This shows that in general, the fragmentation-vanishing diffusion process requires very strong 
control (large $\ell_p$)  of the initial distribution of massive particles as $p\to\infty$.

\section{The local well-posedness}\label{sec4}

In this section, we prove Theorem~\ref{lm2.2}. Our main tool here is Theorem~\ref{lm2.1} 
that turns \eqref{eq1.1} into an abstract semi-linear parabolic problem.

\subsection{Auxiliary estimates}\label{sec4.1}

Proof of Theorem~\ref{lm2.2} relies on several basic facts which follow directly from 
Theorem~\ref{lm2.1} and assumptions listed in Subsection~\ref{sec2.2}. 
\begin{lemma}\label{lm4.2}
The semigroups 
$\bigl\{e^{t\mathcal{L}_{p,\ell}}\bigr\}_{t\ge 0}$ of Theorem~\ref{lm2.1} are compatible in the sense that 
\begin{subequations}\label{eq4.2}
\begin{align}
\label{eq4.2a}
&e^{t\mathcal{L}_{p_1,\ell_1}} u = e^{t\mathcal{L}_{p_2,\ell_2}} u,\quad 
u\in X^{p_1}_{\ell_1} \cap X^{p_2}_{\ell_2},\\
\label{eq4.2b}
&1\le p_1,p_2<\infty,
\quad \bar{\ell}_0\vee \ell_{p_1}<\ell_1,\; \bar{\ell}_0\vee \ell_{p_2}<\ell_2.
\end{align}
\end{subequations}
\end{lemma}
\begin{proof}
The semigroups $\bigl\{e^{t\mathcal{T}_{p,\ell}}\bigr\}_{t\ge 0}$ are compatible, for by virtue 
of Theorem~\ref{lm3.3}(ii) they are given by the same integral formula \eqref{eq3.6b}. 
This observation, together with arguments from parts (a) and (b) of the proof of Theorem~\ref{lm2.1} 
(see Subsection~\ref{sec3.3}), implies that 
\[
(z - \mathcal{L}_{p_1,\ell_1})^{-1} u 
= (z - \mathcal{L}_{p_2,\ell_2})^{-1} u,\quad u\in X^{p_1}_{\ell_1} \cap X^{p_2}_{\ell_2},
\] 
for all $z \in \{\lambda\}+\{|\Arg z|<\theta\}$ and some 
$\tfrac{\pi}{2}<\theta<\pi$, provided $\lambda>0$ is sufficiently large.
Using the latter identity, together with the Dunford integral representations of 
$e^{t\mathcal{L}_{p_i,\ell_i}} u$, $i=1,2$, we arrive at \eqref{eq4.2}. 
\end{proof}

\begin{lemma}\label{lm4.3}
Under assumptions of Theorem~\ref{lm2.1},
\begin{subequations}\label{eq4.3}
\begin{align}
\label{eq4.3a}
&e^{\cdot \mathcal{L}_{p,\ell}}\in \mathcal{L}\bigl(X^p_{\ell}, 
C_{s+\frac{d}{2\delta}(\frac{1}{p}-\frac{1}{q})}\bigl((0,T], X^q_{\ell,s}\bigr)\bigr),\\
\label{eq4.3b}
&0< s+\tfrac{d}{2\delta}\bigl(\tfrac{1}{p}-\tfrac{1}{q}\bigr)\le 1,\quad 1\le p\le q<\infty,
\quad \bar{\ell}_0\vee\ell_p<\ell,
\end{align}
\end{subequations}
for every $0<T<\infty$. 
\end{lemma}
\begin{proof} This is an elementary consequence of \eqref{eq2.7}. 
Indeed, $\mathcal{D}(\mathbb{R}^d\times\mathbb{R}_+)$ is the common core of 
$\bigl(\mathcal{L}_{p,\ell}, D(\mathcal{T}_{p,\ell})\bigr)$, $1\le p<\infty$,
hence, for $u_0\in X^p_\ell$ and $\varepsilon>0$ being fixed, one can choose 
$v_0\in \mathcal{D}(\mathbb{R}^d\times\mathbb{R}_+)$ 
so that
\[
\|u_0-v_0\|_{X^p_\ell} \sup_{t\in(0,T]} t^{s+\frac{d}{2\delta}(\frac{1}{p}-\frac{1}{q})} 
\|e^{t\mathcal{L}_{p,\ell}}\|_{X^p_\ell \to X^q_{\ell,s}} < \varepsilon.
\]
Subsequently,
\begin{align*}
\limsup_{t\to 0^+} t^{s+\frac{d}{2\delta}(\frac{1}{p}-\frac{1}{q})} \|e^{t\mathcal{L}_{p,\ell}}u_0\|_{X^{q}_{\ell,s}} 
&\le \|u_0-v_0\|_{X^p_\ell} \sup_{t\in(0,T]} t^{s+\frac{d}{2\delta}(\frac{1}{p}-\frac{1}{q})} 
\|e^{t\mathcal{L}_{p,\ell}}\|_{X^p_\ell \to X^p_{\ell,s}}\\
&+\limsup_{t\to 0^+} t^{s+\frac{d}{2\delta}(\frac{1}{p}-\frac{1}{q})} \|e^{t\mathcal{L}_{p,\ell}}v_0\|_{X^{q}_{\ell,s}} 
< \varepsilon,
\end{align*} 
and as $\varepsilon$ is arbitrary, we conclude that 
\[
\lim_{t\to 0^+} t^{s+\frac{d}{2\delta}(\frac{1}{p}-\frac{1}{q})} \|e^{t\mathcal{L}_{p,\ell}}u_0\|_{X^{q}_{\ell,s}} = 0.
\]
Further, when $0<t_1\le t_2\le T$, letting $r:=s+\tfrac{d}{2\delta}\bigl(\tfrac{1}{p}-\tfrac{1}{q}\bigr)$ for brevity,
we infer
\begin{align*}
\|t_2^r e^{t_2\mathcal{L}_{p,\ell}}u_0 
- t_1^r e^{t_1\mathcal{L}_{p,\ell}}u_0 \|_{X^{q}_{\ell,s}}
&\le t_1^r 
\|e^{t_1\mathcal{L}_{p,\ell}}\|_{X^p_{\ell}\to X^q_{\ell,s}}
 \bigl\|(e^{(t_2-t_1)\mathcal{L}_{p,\ell}} - I)u_0\bigr\|_{X^p_\ell}\\
&+ \tfrac{t_2^r-t_1^r}{t_2^r}
\|e^{t_2\mathcal{L}_{p,\ell}}\|_{X^p_\ell\to X^q_{\ell,s}}\|u_0\|_{X^p_\ell}.
\end{align*}
The last bound, together with \eqref{eq2.7} and strong continuity of 
$\bigl\{e^{t\mathcal{L}_{p,\ell}}\bigr\}_{t\ge 0}$, gives the inclusion
$t^{s+\frac{d}{2\delta}(\frac{1}{p}-\frac{1}{q})}e^{t\mathcal{L}_{p,\ell}}u_0\in C\bigl((0,T], X^q_{\ell,s}\bigr)$.
\end{proof}

\begin{lemma}\label{lm4.4}
Under assumption \eqref{eq2.6}, 
the map $\mathcal{C}(u,v)$ defined in \eqref{eq1.1c} is symmetric, bi-linear 
and satisfies
\begin{align}
\nonumber
\bigl\|\mathcal{C}(u,v)\bigr\|_{X^p_{\ell}} 
&\le \tfrac{c_\varkappa}{2}(1+2^\ell) \bigl[
\|u\|_{X^{pq_1}_{\ell,\rho}} \|v\|_{X^{pq_1'}}
+\|u\|_{X^{pq_2}} \|v\|_{X^{pq_2'}_{\ell,\rho}}\\
\label{eq4.4}
&\qquad\qquad\quad
+\|u\|_{X^{pq_3}_{\ell}} \|v\|_{X^{pq_3'}_{0,\rho}}
+\|u\|_{X^{pq_4}_{0,\rho}} \|v\|_{X^{pq_4'}_{\ell}}
\bigr],
\end{align}
for any $1\le p<\infty$, $\ell\ge 1$, $1\le q_i\le \infty$, $1\le i \le 4$, 
provided all quantities appearing in the right-hand side of \eqref{eq4.4} are finite. 
\end{lemma}
\begin{proof}
Bound \eqref{eq4.4} follows 
immediately from \eqref{eq1.1c}, \eqref{eq2.6} and the inequality 
$(\xi+\eta)^\ell \le 2^{\ell}(\xi^\ell+\eta^\ell)$, $\ell\ge 1$, $\xi,\eta\in\mathbb{R}_+$.
\end{proof}

In the absence of $X^\infty_\ell$ estimates, our proof of Theorem~\ref{lm2.2}(iii)
relies on solving non-autonomous parabolic problem of the form 
\begin{equation}\label{eq4.5}
u_ t  = (\mathcal{L}_{p,\ell} + \mathcal{E}_{v}) u,\quad 
u(\tau) = u_0, \quad 
0\le \tau<t\le T,
\end{equation}
where 
\[
\mathcal{E}_{v}[u] := - \bigl[(1+\underline{\beta})^\rho v\bigr] u
\]
is the multiplication operator
and $v$ is a H\"older continuous $L^{2p}(\mathbb{R}^d)$-valued function defined
in $[0,T]$. In connection with \eqref{eq4.5}, we have
\begin{lemma}\label{lm4.5}
Assume 
\begin{align*}
&2\le p,\quad \tfrac{d}{4(1-\rho)\delta}<p,\quad \bar{\ell}_0\vee \ell_p < \ell,\\
&v\in C^\mu \bigl([0,T], L^{2p}_+(\mathbb{R}^d)\bigr), \quad 0<\mu <1.
\end{align*}
Then there exists an evolution family $\bigl\{\mathcal{U}_{v,p,\ell}(t,\tau)\bigr\}_{0\le \tau\le t\le T}$
of positive, bounded, linear maps that: 
\begin{itemize}
\item[(i)] satisfies
\begin{align}
\label{eq4.6}
&\|\mathcal{U}_{v,p,\ell}(t,\tau)\|_{X^p_{\ell}\to X^q_{\ell,s}} 
\le C_{v, p,q, T,s} (t-\tau)^{-s-\frac{d}{2\delta}(\frac{1}{p}-\frac{1}{q})},\\
\nonumber
&\tfrac{d}{2(1-\rho)\delta}<p\le q<\infty,
\quad 0\le s \le 1 - \tfrac{d}{2\delta}\bigl(\tfrac{1}{p}-\tfrac{1}{q}\bigr),
\end{align}
with some $C_{v, p,q, T,s}\ge 1$ that depends on $v$, $p$, $q$, $T$, $s$, $\ell$, $\kappa(\alpha)$ and 
$\kappa_\delta(\alpha, \beta)$ only;
\item[(ii)] solves \eqref{eq4.5} in the classical sense, i.e. 
\begin{align*}
\mathcal{U}_{v,p,\ell}(\cdot,\tau)u_0 
&\in 
C\bigl([\tau, T], X^p_\ell\bigr)
\cap
C^1\bigl((\tau, T], X^p_\ell \bigr)
\cap 
C\bigl((\tau, T], D(\mathcal{T}_{p,\ell})\bigr)
\end{align*}
and $u(\cdot) := \mathcal{U}_{v,p,\ell}(\cdot,\tau)u_0$ satisfies \eqref{eq4.5} in $X^p_\ell$. 
\end{itemize}
\end{lemma}
\begin{proof}
(a) Let $p$ and $v$ be as above and let some $0\le t\le T$ be fixed. First, 
we note that the maximal $X^p_\ell$ realization of $\mathcal{E}_{v}(t)$,
\[
\bigl(\mathcal{E}_{v,p,\ell}(t), D\bigl(\mathcal{E}_{v,p,\ell}(t)\bigr)\bigr) ,\quad 
D\bigl(\mathcal{E}_{v,p,\ell}(t)\bigr) := \bigl\{u\in X^p_\ell\,\bigr|\, 
[(1+\underline{\beta})^\rho v(t)] u\in X^p_{\ell}\bigr\},
\]
satisfies $D(\mathcal{T}_{p,\ell})\hookrightarrow X^p_\ell\cap X^{2p}_{\ell,\rho} 
\hookrightarrow D\bigl(\mathcal{E}_{v,p,\ell}(t)\bigr)$. 
This follows directly from H\"older's inequality and \eqref{eq2.7},
for if $\lambda>\omega_{p,\ell}>0$ is large,
\begin{align}
\nonumber
&\bigl\|\mathcal{E}_{v,p,\ell}(t)(\lambda - \mathcal{L}_{p,\ell})^{-1} \bigr\|_{X^p_\ell\to X^p_\ell}
\le \|v(t)\|_{L^{2p}(\mathbb{R}^d)} 
\bigl\|(\lambda - \mathcal{L}_{p,\ell})^{-1}\bigr\|_{X^p_{\ell}\to X^{2p}_{\ell,\rho}}\\
\nonumber
&\qquad\qquad
\le C_{p,q,\rho}\|v(t)\|_{L^{2p}(\mathbb{R}^d)} 
\int_{\mathbb{R}^+} e^{-(\lambda-\omega_{p,\ell})t}t^{-\rho-\frac{d}{4p\delta}}dt\\
\label{eq4.7}
&\qquad\qquad
= \Gamma\bigl(1-\rho-\tfrac{d}{4p\delta}\bigr) C_{p,q,\rho}\|v(t)\|_{L^{2p}(\mathbb{R}^d)}
(\lambda-\omega_{p,\ell})^{-1+\rho+\tfrac{d}{4p\delta}}. 
\end{align}
It follows from \eqref{eq4.7} that 
$\bigl(\mathcal{L}_{p,\ell} \pm \mathcal{E}_{v,p,\ell}(t), D(\mathcal{T}_{p,\ell})\bigr)$ is closed and 
that for large values of $\lambda>\omega_{p,\ell}$, the resolvent
$\bigl(\lambda - \mathcal{L}_{p,\ell} + \mathcal{E}_{v,p,\ell}(t)\bigr)^{-1}\subset \mathcal{L}(X^p_\ell)$
is given by a positive and convergent Neumann series. Since 
$\bigl(\mathcal{L}_{p,\ell},D(\mathcal{T}_{p,\ell})\bigr)$ itself generates a positive analytic 
$C_0$-semigroup and since $-\mathcal{E}_{v,p,\ell}(t)$ is positive 
in $X^p_{\ell,+}\cap D\bigl(\mathcal{E}_{v,p,\ell}(t)\bigr)$, 
\cite[Theorem 1.1]{ArRh1991} implies that either of the operators
$\bigr(\mathcal{L}_{p,\ell}\pm \mathcal{E}_{v,p,\ell}(t),D(\mathcal{T}_{p,\ell})\bigr)$ generate 
an analytic $C_0$-semigroup in $X^p_\ell$.

To show that the semigroup $\bigl\{e^{\tau(\mathcal{L}_{p,\ell}+\mathcal{E}_{v,p,\ell}(t))}\bigr\}_{\tau\ge 0}$
is positive, we note that $\bigl(\mathcal{E}_{v,p,\ell}(t), D\bigl(\mathcal{E}_{v,p,\ell}(t)\bigr)\bigr) $ is sectorial, 
with $\sigma\bigl(\mathcal{E}_{v,p,\ell}(t)\bigr) = \bar{\mathbb{R}}_-$, and that
the associated semigroup $\bigl\{e^{\tau \mathcal{E}_{v,p,\ell}(t)}\bigl\}_{\tau\ge 0}$
is positive, strongly continuous and analytic in $X^p_\ell$. 
It is obvious that $|e^{\tau \mathcal{E}_{v,p,\ell}(t)} u| \le |u|$. Therefore, 
the positivity of $\bigl\{e^{\tau \mathcal{L}_{p,\ell}}\bigl\}_{\tau\ge 0}$, yields the point-wise bound
\[
\Bigl|\bigl[e^{\frac{\tau}{n}\mathcal{L}_{p,\ell}} e^{\frac{\tau}{n}\mathcal{E}_{v,p,\ell}(t)}\bigr]^n u\bigl| 
\le e^{\tau\mathcal{L}_{p,\ell}} |u|,\quad u\in X^p_\ell,\quad \tau\ge 0
\]
and, in particular, $\mathcal{L}(X^p_\ell)$-valued products 
$V_n(\tau) :=\bigl[e^{\frac{\tau}{n}\mathcal{L}_{p,\ell}} e^{\frac{\tau}{n}\mathcal{E}_{v,p,\ell}(t)}\bigr]^n$, 
$n\ge 0$, are positive and uniformly bounded on compact subintervals of $\bar{\mathbb{R}}_+$. 
By virtue of the classical Lie-Trotter approximation formula (see e.g. 
\cite[Corollary~5.8, p. 227]{EnNa2000}) and by the closeness on the positive cone 
$X^p_{\ell,+}$, $\bigl\{V_n(\cdot)\bigr\}_{n\ge 1}$ converges on compact subintervals 
of $\bar{\mathbb{R}}_+$ to a positive $C_0$-semigroup, 
generated by $\overline{\bigl(\mathcal{L}_{p,\ell} + \mathcal{E}_{v,p,\ell}(t), D(\mathcal{T}_{p,\ell})\bigr)} 
= \bigl(\mathcal{L}_{p,\ell} + \mathcal{E}_{v,p,\ell}(t), D(\mathcal{T}_{p,\ell})\bigr)$.

(b) Since $v$ is assumed to be H\"older continuous, it follows from \eqref{eq4.7} that
$\mathcal{L}_{p,\ell} + \mathcal{E}_{v,p,\ell}(\cdot)\in 
C^\mu\bigl([0,T], \mathcal{L}\bigl(D(\mathcal{T}_{p,\ell}), X^p_\ell\bigr)\bigr)$.
The last inclusion and part (a) of the proof are sufficient 
(see e.g. \cite[Corollary~4.4.2, p. 66]{Am2011} or 
\cite[Section~6.1, pp. 212-228]{Lun1995}) to guarantee
existence of the evolution operator $\mathcal{U}_{v,p,\ell}$ asserted in the Lemma.
Its positivity follows from the positivity of individual semigroups 
$\bigl\{e^{\tau(\mathcal{L}_{p,\ell}+\mathcal{E}_{v,p,\ell}(t))}\bigr\}_{\tau\ge 0}$, $0\le t\le T$, 
see e.g. \cite[Theorem~6.4.3, pp. 83-84]{Am2011}. 
Finally, \eqref{eq4.6} is the consequence of the embedding 
$D(\mathcal{T}_{p,\ell})\hookrightarrow X^{q}_{\ell,s}$, see \eqref{eq3.7f}; the standard bound 
\[
\|\mathcal{U}_{p,v,\ell}(t,\tau)\|_{X^p_\ell\to D(\mathcal{T}_{p,\ell})} \le C_{v,p,q, T}(t-\tau)^{-1}, 
\quad 0\le\tau<t\le T,
\]
see e.g. \cite[Corollary~6.1.8, pp. 219-222]{Lun1995};
and the interpolation identity \eqref{eq3.13}, mentioned in the proof of Theorem~\ref{lm2.1}.
\end{proof} 

In subsequent calculations, we make a consistent use of the fractional/sin\-gular version 
of Gronwall's inequality, see e.g.
\cite[Theorem 3.3.1, p. 52]{Am2011} or \cite[Exercise 3, p. 190]{Hen1993}.
\begin{lemma}\label{lm4.6}
Assume $f\in L^\infty_{\text{loc},+}(\mathbb{R}_+)$ satisfies
\begin{subequations}\label{eq4.8}
\begin{equation}\label{eq4.8a}
f(t) \le a t^{-\mu} + b\int_0^t(t-\tau)^{-\nu} f(\tau) d\tau,\quad\text{a.e. in $\mathbb{R}_+$},
\end{equation}
with $a,b>0$ and $0\le \mu,\nu<1$. Then 
\begin{equation}\label{eq4.8b}
f(t) \le c_0 a t^{-\mu} e^{c_1 t},\quad\text{a.e. in $\mathbb{R}_+$}, 
\end{equation}
\end{subequations}
where $c_0,c_1> 0$ depend on the parameters $b$, $\mu$ and $\nu$ of \eqref{eq4.8a}.
\end{lemma}

\subsection{Mild solutions}\label{sec4.2} 

In this Subsection, we establish existence of unique mild solutions 
to \eqref{eq1.1}, i.e. solutions to the Volterra integral equation
\begin{equation}\label{eq4.9}
u(t) = e^{t\mathcal{L}_{p,\ell}}u_0 + \int_0^t e^{(t-\tau)\mathcal{L}_{p,\ell}} 
\mathcal{C}\bigl(u(\tau),u(\tau)\bigr)d\tau, \quad u_0\in Y^p_{\ell},
\end{equation} 
for sufficiently large $\ell$ and $1<p<\infty$.
In connection with our choice $Y^p_{\ell}$ for the phase space , we remark that the coagulation
operator $\mathcal{C}(u,v)$, defined in \eqref{eq1.1c}, is locally Lipschitz continuous 
from $X^{qp}_{\ell,\rho}\times X^{q'p}_{\ell,\rho}$ to $X^p_{\ell}$, $1\le p,q\le \infty$, 
and hence local solvability of \eqref{eq4.9} relies on the hyper-contractivity 
(estimate \eqref{eq2.7}) of the semigroups $\bigl\{e^{t\mathcal{L}_{p,\ell}}\bigr\}_{t\ge 0}$.  
In particular, $X^1_{\ell}$-data requires $X^\infty_{\ell,\rho}$ estimates, which in turn 
require $1\le d < 2(1-\rho)\delta$. The last condition is very restrictive, and 
we are forced to impose stronger integrability assumptions on the input data.  

\begin{proof}[Proof of Theorem~\ref{lm2.2}(i)]
(a) For $u_0\in Y^p_{\ell}$ and $0<T$, we let 
$v_0(t) = e^{t\mathcal{L}_{p,\ell}} u_0$, $0\le t\le T$, and observe that on the account of 
Theorem~\ref{lm2.1} and Lemma~\ref{lm4.3},
$v_0\in Z^p_{\ell,\rho}(T)$. We define further
\begin{align*}
&B_T(u_0) := \bigl\{v\in Z^p_{\ell,\rho}(T)\,\bigr|\, \|v-v_0\|_{Z^p_{\ell,\rho}(T)}\le 1\bigr\},\\
&\mathcal{F}[v](t) := v_0(t) + \int_{0}^t e^{(t-\tau)\mathcal{L}_{p,\ell}} 
\mathcal{C}\bigl(v(\tau),v(\tau)\bigr)d\tau,\quad v\in Z^{p}_{\ell,s}(T).
\end{align*}

(b) In view of \eqref{eq2.7}, Lemmas~\ref{lm4.2}-\ref{lm4.4} and  after some elementary 
calculations, we infer
\begin{align*}
&t^{j\rho} \bigl\|\mathcal{F}[v](t) - v_0(t)\bigr\|_{X^{p}_{\ell,j\rho}} 
\le t^{j\rho}\int_0^t \|e^{(t-\tau)\mathcal{L}_{\frac{p}{2},\ell}}\|_{X^{\frac{p}{2}}_\ell \to X^{p}_{\ell,j\rho}} 
\bigl\|\mathcal{C}\bigl(v(\tau),v(\tau)\bigr)\bigr\|_{X^{\frac{p}{2}}_\ell}d\tau\\
&\qquad
\le c_1' t^{j\rho} \int_0^t (t-\tau)^{-j\rho - \tfrac{d}{2p\delta}}\tau^{-\rho} e^{\omega_{\frac{p}{2},\ell} (t-\tau)} 
\bigl\|v(\tau)\bigr\|_{X^p_{\ell}} \bigl[\tau^\rho \bigl\|v(\tau)\bigr\|_{X^p_{\ell,\rho}}\bigr] d\tau\\
&\qquad
\le c_1 e^{\omega_{\frac{p}{2},\ell} t} t^{1-\rho-\tfrac{d}{2p\delta}} \bigl(\|u_0\|_{Y^p_{\ell}}+1\bigr)^2,
\quad j=0,1,
\end{align*}
where the constant $c_1\ge 1$ depends on $p$, $\ell$, $\rho$, $\kappa(\alpha)$, 
$\kappa_\delta(\alpha,\beta)$ and $c_\varkappa$ only. 
Similarly, using \eqref{eq2.7}, Lemmas~\ref{lm4.2}-\ref{lm4.4} and the elementary bound
\[
\|u\|_{X^2_{\ell,j\rho}} \le \|u\|_{X^1_{\ell,j\rho}}^{1-\frac{p'}{2}}\|u\|_{X^p_{\ell,j\rho}}^{\frac{p'}{2}},
\quad 2\le p<\infty,\quad j=0,1,
\] 
we obtain
\begin{align*}
&t^{j\rho}\bigl\|\mathcal{F}[v](t) - v_0(t)\bigr\|_{X^{1}_{\ell,j\rho}} 
\le t^{j\rho} \int_0^t \|e^{(t-\tau)\mathcal{L}_{1,\ell}}\|_{X^{1}_\ell \to X^{1}_{\ell,j\rho}} 
\bigl\|\mathcal{C}\bigl(v(\tau),v(\tau)\bigr)\bigr\|_{X^{1}_\ell}d\tau\\
&\qquad
\le c_2' t^{j\rho} 
\int_0^t (t-\tau)^{-j\rho} e^{\omega_{1,\ell} (t-\tau)} \bigl\|v(\tau)\bigr\|_{X^2_{\ell}} 
\bigl\|v(\tau)\bigr\|_{X^2_{\ell,\rho}} d\tau\\
&\qquad
\le c_2' t^{j\rho} \int_0^t (t-\tau)^{-j\rho} \tau^{-\rho} e^{\omega_{1,\ell} (t-\tau)} 
\bigl\|v(\tau)\bigr\|^{1-\frac{p'}{2}}_{X^1_{\ell}} 
\bigl\|v(\tau)\bigr\|^{\frac{p'}{2}}_{X^p_{\ell}}\\
&\qquad\qquad\qquad\qquad\qquad\qquad\qquad\qquad
\bigl[\tau^\rho \bigl\|v(\tau)\bigr\|_{X^1_{\ell,\rho}}\bigr]^{1-\frac{p'}{2}}
\bigl[\tau^\rho \bigl\|v(\tau)\bigr\|_{X^p_{\ell,\rho}}\bigr]^{\frac{p'}{2}} d\tau\\
&\qquad
\le c_2 e^{\omega_{1,\ell} t} t^{1-\rho} \bigl(\|u_0\|_{Y^p_{\ell}}+1\bigr)^2,
\quad j=0,1,
\end{align*}
with the constant $c_2\ge 1$ controlled by the same quantities as $c_1$.
Finally, using \eqref{eq2.7}, Lemmas~\ref{lm4.2}-\ref{lm4.4} and the inequality
\[
\|u\|_{X^{\frac{4p}{3}}_{\ell,j\rho}} \le \|u\|_{X^p_{\ell,j\rho}}^{\frac{1}{2}}\|u\|_{X^{2p}_{\ell,j\rho}}^{\frac{1}{2}},
\quad 1\le p<\infty,\quad j=0,1,
\] 
we have
\begin{align*}
&t^{j\rho+\frac{d}{4p\delta}} \bigl\|\mathcal{F}[v](t) - v_0(t)\bigr\|_{X^{2p}_{\ell,j\rho}} \\
&\qquad
\le t^{j\rho+\frac{d}{4p\delta}} 
\int_0^t \|e^{(t-\tau)\mathcal{L}_{\frac{2p}{3},\ell}}\|_{X^{\frac{2p}{3}}_\ell \to X^{2p}_{\ell,j\rho}} 
\bigl\|\mathcal{C}\bigl(v(\tau),v(\tau)\bigr)\bigr\|_{X^{\frac{2p}{3}}_\ell}d\tau\\
&\qquad
\le c_3' t^{j\rho+\frac{d}{4p\delta}} 
\int_0^t (t-\tau)^{-j\rho - \frac{d}{2p\delta}} e^{\omega_{\frac{2p}{3},\ell} (t-\tau)} 
\bigl\|v(\tau)\bigr\|_{X^{\frac{4p}{3}}_{\ell}} \bigl\|v(\tau)\bigr\|_{X^{\frac{4p}{3}}_{\ell,\rho}} d\tau\\
&\qquad
\le c_3' t^{j\rho+\frac{d}{4p\delta}}
\int_0^t (t-\tau)^{-j\rho - \frac{d}{2p\delta}} \tau^{-\rho-\tfrac{d}{4p\delta}} e^{\omega_{\frac{2p}{3},\ell} 
(t-\tau)} \bigl\|v(\tau)\bigr\|^{\frac{1}{2}}_{X^p_{\ell}} 
\bigl[\tau^\rho \bigl\|v(\tau)\bigr\|_{X^p_{\ell,\rho}}\bigr]^{\frac{1}{2}}\\
&\qquad\qquad\qquad\qquad\qquad\qquad\qquad\qquad
\bigl[\tau^{\frac{d}{4p\delta}} \bigl\|v(\tau)\bigr\|_{X^{2p}_{\ell}}\bigl]^{\frac{1}{2}}
\bigl[\tau^{\rho+\tfrac{d}{4p\delta}} \bigl\|v(\tau)\bigr\|_{X^{2p}_{\ell,\rho}}\bigr]^{\frac{1}{2}} d\tau\\
&\qquad
\le c_3 e^{\omega_{\frac{2p}{3},\ell} t} t^{1-\rho-\frac{d}{2p\delta}} \bigl(\|u_0\|_{Y^p_{\ell}}+1\bigr)^2,
\quad j=0,1,
\end{align*}
with the constant $c_3\ge 1$ controlled by the same quantities as $c_1$ and $c_2$.
The estimates above indicate that  
$\mathcal{F}[B_T(u_0)]\subset B_T(u_0)$, provided $T \le 1\wedge T_1\wedge T_2$, 
where
\begin{align*}
&T_1:=\bigl[4 (c_1\vee c_3) e^{\omega_{\frac{p}{2},\ell}\vee \omega_{\frac{2p}{3},\ell}} 
\bigl(\|u_0\|_{Y^p_{\ell}}+1\bigr)^2\bigr]^{-\frac{2p\delta}{2p\delta(1-\rho)-d}},\\
&T_2:=\bigl[4 c_2 e^{\omega_{1,\ell}} 
\bigl(\|u_0\|_{Y^p_{\ell}}+1\bigr)^2\bigr]^{-\frac{1}{1-\rho}}.
\end{align*}

(c) When $v_1, v_2\in B_T(u_0)$, we have also
\begin{align*}
&t^{j\rho}\bigl\|\mathcal{F}[v_1](t) - F[v_2](t)\bigr\|_{X^{p}_{\ell,j\rho}} \\
&\qquad
\le t^{j\rho}\int_0^t \|e^{(t-\tau)\mathcal{L}_{\frac{p}{2},\ell}}\|_{X^{\frac{p}{2}}_\ell \to X^{p}_{\ell,j\rho}} 
\bigl\|\mathcal{C}\bigl(v_1(\tau)-v_2(\tau), v_1(\tau)+v_2(\tau)\bigr)\bigr\|_{X^{\frac{p}{2}}_\ell}d\tau\\
&\qquad
\le c_4' \tau^{j\rho} \int_0^t (t-\tau)^{-j\rho - \tfrac{d}{2p\delta}}\tau^{-\rho} 
e^{\omega_{\frac{p}{2},\ell} (t-\tau)}\\ 
&\qquad\qquad\qquad\qquad \qquad 
\Bigl[
\bigl[\tau^\rho \bigl\|v_1(\tau)-v_2(\tau)\bigr\|_{X^p_{\ell,\rho}}\bigr] 
\bigl\|v_1(\tau)+v_2(\tau)\bigr\|_{X^p_{\ell}} \\
&\qquad\qquad\qquad\qquad\qquad 
+\bigl\|v_1(\tau)-v_2(\tau)\bigr\|_{X^p_{\ell}} 
\bigl[\tau^\rho \bigl\|v_1(\tau)+v_2(\tau)\bigr\|_{X^p_{\ell,\rho}}\bigr]
\Bigr]d\tau\\
&\qquad
\le c_4 e^{\omega_{\frac{p}{2},\ell} t} t^{1-\rho-\tfrac{d}{2p\delta}} \bigl(\|u_0\|_{Y^p_{\ell}}+1\bigr)
\|v_1-v_2\|_{Z^p_{\ell,\rho}(T)},\quad j=0,1,
\end{align*}
and then, after similar calculations,
\begin{align*}
&t^{j\rho}\bigl\|\mathcal{F}[v_1](t) - F[v_2](t)\bigr\|_{X^{1}_{\ell,j\rho}} \\
&\qquad
\le c_5 e^{\omega_{1,\ell} t} t^{1-\rho} \bigl(\|u_0\|_{Y^p_{\ell}}+1\bigr)
\|v_1-v_2\|_{Z^p_{\ell,\rho}(T)},\quad j=0,1,\\
&t^{j\rho+\frac{d}{4p\delta}}\bigl\|\mathcal{F}[v_1](t) - F[v_2](t)\bigr\|_{X^{2p}_{\ell,j\rho}} \\
&\qquad
\le c_6 e^{\omega_{\frac{2p}{3},\ell} t} t^{1-\rho-\frac{d}{2p\delta}} \bigl(\|u_0\|_{Y^p_{\ell}}+1\bigr)
\|v_1-v_2\|_{Z^p_{\ell,\rho}(T)},\quad j=0,1.
\end{align*}
It follows that $\mathcal{F}$ is a strict contraction in $B_T(u_0)$, provided
$T< 1\wedge \Bigl(\bigwedge_{i=1}^4 T_i\Bigr)$, where 
\begin{align*}
&T_3:=\bigl[4(c_4\vee c_6) e^{\omega_{\frac{p}{2},\ell}\vee \omega_{\frac{3p}{2},\ell}} 
\bigl(\|u_0\|_{Y^p_{\ell}}+1\bigr)\bigr]^{-\frac{2p\delta}{2p\delta(1-\rho)-d}},\\ 
&T_4:=\bigl[4 c_4 e^{\omega_{1,\ell}} 
\bigl(\|u_0\|_{Y^p_{\ell}}+1\bigr)\bigr]^{-\frac{1}{1-\rho}}.
\end{align*} 
Since $B_T(u_0)$ is closed in $Z^p_{\ell,\rho}(T)$ and the latter space is complete, 
the assertion follows from the classical Banach fixed point theorem.
\end{proof}

\begin{remark}\label{lm4.7}
To conclude this section, we note that in view of \eqref{eq3.7f} and \eqref{eq2.8a}, 
\[
D(\mathcal{T}_{p,\ell})\hookrightarrow X^p_{\ell,\rho}\cap X^{2p}_{\ell,\rho}.
\]
As a consequence, for regular data $u_0\in D(\mathcal{T}_{1,\ell})\cap D(\mathcal{T}_{p,\ell})$, 
\[
e^{\cdot \mathcal{L}_{p,\ell}} u_0 \in C\bigl(\bar{\mathbb{R}}_+, D(\mathcal{T}_{1,\ell})
\cap D(\mathcal{T}_{p,\ell})\bigr)
\]
and the mild solution (solution to \eqref{eq4.9}) satisfies
\begin{equation}\label{eq4.10}
u\in C\bigl([0,T'], Y^p_\ell\cap Y^p_{\ell,\rho}\cap X^{2p}_\ell\cap X^{2p}_{\ell,\rho}\bigr),
\end{equation}
with some $0<T'\le T$.
\end{remark}

\subsection{Regularity}\label{sec4.3} 

Thanks to the special structure of the coagulation operator (see Lemma~\ref{lm4.4}) and
under stronger assumption $\ell>\bar{\ell}_0\vee\ell_{2p}$,
the inclusion \eqref{eq2.8b} yields H\"older's regularity of mild solutions.
\begin{lemma}\label{lm4.8}
Let $0<\varepsilon <T$ be fixed. If $\ell>\bar{\ell}_0\vee\ell_{2p}$,  
the local mild solutions of Theorem~\ref{lm2.2}(i) satisfy
\begin{align}\label{eq4.11}
&u\in C^{\mu}\bigl([\varepsilon, T], Y^{p}_{\ell,\rho}\cap  X^{2p}_{\ell,\rho}\bigr),\quad 
0< \mu <1-\rho.
\end{align}
\end{lemma}
\begin{proof}
(a) Let $u\in Z^p_{\ell,\rho}(T)$ be the mild solution constructed in Section~\ref{sec4.2}. 
For $0<\varepsilon<T$ and $0<h<T-\varepsilon$, we set
\begin{align*}
&v_0 := u(\varepsilon),\quad v(t) := u(t+\varepsilon),\quad 0\le t \le T-\varepsilon,\\
&\Delta_h^\mu [v](t) = \tfrac{v(t+h) - v(t)}{h^\mu},\quad 0< \mu < 1, \quad \varepsilon\le t \le T-\varepsilon-h
\end{align*}
and observe that on the account of \eqref{eq2.8b},
\begin{align*}
v\in Z^p_{\ell,\rho}(\varepsilon, T) & :=C\bigl([0,T-\varepsilon], Y^p_\ell \cap Y^p_{\ell,\rho}
\cap X^{2p}_\ell\cap X^{2p}_{\ell,\rho}\bigr).
\end{align*}
Using \eqref{eq4.9}, after some rearrangements we infer
\begin{align*}
\Delta_h^\mu[v](t) 
&= h^{-\mu}\int_0^h \mathcal{L}_{p,\ell} e^{(t+\tau)\mathcal{L}_{p,\ell}}d\tau v_0 \\
&+h^{-\mu}\int_0^h e^{(t+h-\tau)\mathcal{L}_{p,\ell}} \mathcal{C}\bigl(v(\tau),v(\tau)\bigr) d\tau\\
&+\int_0^t e^{(t-\tau)\mathcal{L}_{p,\ell}} \mathcal{C}\bigl(\Delta_h^\mu[v](\tau), 
v(\tau+h) + v(\tau)\bigr)d\tau =: I_0(t)+I_1(t) +I_2(t).
\end{align*}
Inclusion \eqref{eq4.11} is a byproduct of this identity and the the elementary bound
\begin{equation}\label{eq4.12}
\sup_{h,t\in\mathbb{R}_+} 
t^\nu \bigl(\tfrac{t}{h}\bigr)^{\mu}\int_0^h \tfrac{d\tau}{(t+\tau)^{1+\nu}} \le c_{\mu,\nu}<\infty,
\; 0< \mu \le 1,\; 0\le \nu.
\end{equation}

(b) We start with $X^{2p}_{\ell,j\rho}$-estimates.
Thanks to the assumption $\ell>\bar{\ell}_0\vee\ell_{2p}$, the semigroup 
$\bigl\{e^{t\mathcal{L}_{2p,\ell}}\bigr\}_{t\ge 0}$ is analytic in $X^{2p}_\ell$. 
Therefore, using \eqref{eq4.12}, \eqref{eq2.7} and the compatibility Lemma~\ref{lm4.2}, we infer
\begin{align*}
&\bigl\|I_0(t)\bigr\|_{X^{2p}_{\ell,j\rho}} \le c_1' h^{-\mu }\int_0^h(t+\tau)^{-1-j\rho} d\tau
\|v_0\|_{X^{2p}_{\ell, j\rho}}
\le c_1 t^{-j\rho - \mu }\|v_0\|_{X^{2p}_{\ell, j\rho}},\\ 
&0< \mu\le 1,\quad j=0,1,
\end{align*}
where $c_1$ depends on the semigroup $\bigl\{e^{t\mathcal{L}_{p,\ell}}\bigr\}_{t\ge 0}$,  
parameters $\mu$, $\rho$ and by the terminal time $T>0$ only.
Bound \eqref{eq2.7}, together with Lemma~\ref{lm4.4} and the inclusion 
$v\in Z^p_{\ell,\rho}(\varepsilon, T)$, gives
\begin{align*}
\bigl\|I_1(t)\bigr\|_{X^{2p}_{\ell,j\rho}} 
&\le c_2' h^{-\mu}\int_0^h (t+\tau)^{-\rho j - \frac{d}{4p\delta}}d\tau
\|v\|^2_{Z^{p}_{\ell,\rho}(\varepsilon,T)}\\
&\le c_2 \|v\|^2_{Z^{p}_{\ell,\rho}(\varepsilon,T)} t^{-j\rho-\mu}, 
\quad 0< \mu \le 1, \quad j=0,1, 
\end{align*}
where $c_2\ge 1$ is controlled by the same quantities as $c_1$.
Finally, using \eqref{eq2.7}, Lemma~\ref{lm4.4} and the inclusion $v\in Z^p_{\ell,\rho}(\varepsilon, T)$,
we have also
\begin{align*}
&\bigl\|I_2(t)\bigr\|_{X^{2p}_{\ell,j\rho}} 
\le c_3\|v\|_{Z^p_{\ell,\rho}(\varepsilon, T)}  
\int_0^t (t-\tau)^{-j\rho-\frac{d}{4p\delta}} 
\bigl\|\Delta^\mu_h[v](\tau)\bigr\|_{X^{2p}_{\ell,\rho}} d\tau,\\
&0<\mu\le 1, \quad j=0,1.
\end{align*}

We assemble all our estimates together to obtain
\begin{align}\label{eq4.13}
\bigl\|\Delta^\mu_h[v](t)\bigr\|_{X^{2p}_{\ell,\rho}} \le C_1 t^{-\rho -  \mu}
+ C_2 \int_0^t (t-\tau)^{-\rho-\frac{d}{4p\delta}} 
\bigl\|\Delta^\mu_h[v](\tau)\bigr\|_{X^{2p}_{\ell,\rho}}d\tau,
\end{align}
where $0<t\le T-\varepsilon - h$ and the constants $C_1, C_2\ge 1$, depend 
on the parameters of model \eqref{eq1.1}, terminal time $T$ and the norm 
$\|v\|_{Z^p_{\ell,\rho}(\varepsilon,T)}$. 
Now direct application of Lemma~\ref{lm4.6} gives the bound
\begin{subequations}\label{eq4.14}
\begin{equation}\label{eq4.14a}
\bigl\|\Delta^\mu_h[v](t)\bigr\|_{X^{2p}_{\ell,\rho}} \le C t^{-\rho - \mu},\quad 0<\mu<1-\rho,\quad 
\quad 0<t<T-\varepsilon - h.
\end{equation}
In view of \eqref{eq2.8a}, direct substitution of \eqref{eq4.14a} into our bound for 
$\bigl\|I_2(t)\bigr\|_{X^{2p}_{\ell}}$ yields also
\begin{equation}\label{eq4.14b}
\bigl\|\Delta^\mu_h[v](t)\bigr\|_{X^{2p}_{\ell}} \le C t^{ - \mu},\quad 0<\mu<1-\rho, 
\quad 0<t<T-\varepsilon - h.
\end{equation}

(c) Calculations, similar to those of part (b) above, give also
\begin{align*}
&\bigl\|I_0(t)\bigr\|_{Y^{p}_{\ell,j\rho}}  \le c_4 \|v_0\|_{Y^{p}_{\ell, j\rho}}t^{-j\rho - \mu },\\
&\bigl\|I_1(t)\bigr\|_{Y^{p}_{\ell,j\rho}} \le c_5  \|v\|^2_{Z^{p}_{\ell,\rho}(\varepsilon,T)} t^{-j\rho-\mu},\\
&\bigl\|I_2(t)\bigr\|_{X^{p}_{\ell,j\rho}} \le c_6 \|v\|_{Z^p_{\ell,\rho}(\varepsilon, T)}
\int_0^t (t-\tau)^{-j\rho-\frac{d}{2p\delta}} \bigl\|\Delta^\mu_h[v](\tau)\bigr\|_{X^p_{\ell,\rho}} d\tau,\\
&\bigl\|I_2(t)\bigr\|_{X^{1}_{\ell,j\rho}} \le c_6 \|v\|_{Z^p_{\ell,\rho}(\varepsilon, T)}
\int_0^t (t-\tau)^{-j\rho} \bigl\|\Delta^\mu_h[v](\tau)\bigr\|_{X^1_{\ell,\rho}}^{1-\frac{p'}{2}} 
\bigl\|\Delta^\mu_h[v](\tau)\bigr\|_{X^p_{\ell,\rho}}^{\frac{p'}{2}} d\tau,
\end{align*}
for $0< \mu \le 1$, $j=0,1$ with $c_4,c_5,c_6\ge 1$ controlled 
by the same quantities as $c_1$, $c_2$ and $c_3$. Using these bounds, 
as in part (b) of the proof we infer
\begin{align}
\label{eq4.14c}
&\bigl\|\Delta^\mu_h[v](t)\bigr\|_{Y^{p}_{\ell,\rho}} \le C t^{-\rho - \mu}, 
\quad 0<t<T-\varepsilon - h,\\
\label{eq4.14d}
&\bigl\|\Delta^\mu_h[v](t)\bigr\|_{Y^{p}_{\ell}} \le C t^{-\mu}, 
\quad 0<t<T-\varepsilon - h,
\end{align}
\end{subequations}
with $0<\mu<1-\rho$.
Bounds \eqref{eq4.14} and our definition of $v(t)$ indicate, that 
\[
u \in C^\mu\bigl([2\varepsilon, T],  Y^p_{\ell,\rho}\cap X^{2p}_{\ell,\rho}\bigr),
\]
for any $0<2\varepsilon<T$ and any exponent $0<\mu<1-\rho$.
The proof is complete.
\end{proof}

As an immediate consequence of Lemma~\ref{lm4.8}, we have
\begin{corollary}\label{lm4.9}
Under assumptions of Theorem~\ref{lm2.2}, for every $0<\varepsilon < T$, 
the map $f: t \mapsto \mathcal{C}\bigl(u(t),u(t)\bigr)$ satisfies
\begin{equation}\label{eq4.15}
f\in C^\mu\bigl([\varepsilon, T], Y^p_{\ell}\bigr),\quad 0\le \mu < 1 - \rho.
\end{equation}
\end{corollary}
\begin{proof}
Since $\Delta_h^\mu [f](t) = \mathcal{C}\bigl(\Delta_h^\mu[u(t)], u(t+h)+u(t)\bigr)$, the result follows 
immediately from Lemma~\ref{lm4.4} and \eqref{eq4.11}.
\end{proof}

\begin{remark}\label{lm4.10}
As before, for smooth data $u_0\in D(\mathcal{T}_{1,\ell})\cap D(\mathcal{T}_{p,\ell})$, the conclusions 
of Lemma~\ref{lm4.8} and Corollary~\ref{lm4.9} can be significantly strengthened. 
Indeed, in this case calculation, similar to those of Lemma~\ref{lm4.8}, give
\begin{align}\label{eq4.16}
&u\in C^\mu\bigl([0,T'], Y^{p}_{\ell,\rho} \cap X^{2p}_{\ell,\rho}\bigr),\quad
f\in C^\mu\bigl([0,T'], Y^p_\ell\bigr),\quad 0\le \mu<1-\rho,
\end{align}
with the same $0<T'\le T$ as in Remark~\ref{lm4.7}.
\end{remark}

Now, we are in the position to settle the second claim of Theorem~\ref{lm2.2}.
\begin{proof}[Proof of Theorem~\ref{lm2.2}(ii)]
(a) Let $0<\varepsilon<T$ be fixed, as in Lemma~\ref{lm4.8}, we denote 
$v_0 := u(\varepsilon)$, $v(t) = u(\varepsilon+t)$, $0\le t\le T-\varepsilon$, and write 
\begin{align*}
v(t) & = e^{t\mathcal{L}_{p,\ell}}v_0 + \int_0^t e^{\tau \mathcal{L}_{p,\ell}} f(t)d\tau
+ \int_0^t e^{(t-\tau) \mathcal{L}_{p,\ell}} \bigl[f(\tau) - f(t)\bigr]d\tau\\
&:= J_0(t) + J_1(t) + J_2(t),\quad 0\le t\le T-\varepsilon. 
\end{align*}
By virtue of Lemma~\ref{lm4.8}, Corollary~\ref{lm4.9} and the analyticity of 
$\bigl\{e^{t\mathcal{L}_{p,\ell}}\bigr\}_{t\ge0}$,
\begin{align*}
&\mathcal{L}_{p,\ell}  J_0(t) \in C\bigl((0, T-\varepsilon], Y^p_\ell\bigr),\\
&\mathcal{L}_{p,\ell} J_1(t)  = (e^{t\mathcal{L}_{p,\ell}} - I) f(t) \in C\bigl((0, T-\varepsilon], Y^p_\ell\bigr).
\end{align*}
Further, using Corollary~\ref{lm4.9} and the analyticity of $\bigl\{e^{t\mathcal{L}_{p,\ell}}\bigr\}_{t\ge0}$,
for a fixed $0<\mu<1-\rho$ we obtain
\[
\bigl\|\mathcal{L}_{p,\ell}  J_2(t)\bigr\|_{Y^p_\ell} \le c \int_0^t (t-\tau)^{\mu-1} d\tau < \infty.
\]
Since the $Y^p_\ell$-valued map $t\mapsto \mathcal{L}_{p,\ell}  J_2(t)$ is continuous,
it follows that 
\begin{equation}\label{eq4.17}
u\in C\bigl([2\varepsilon, T], D(\mathcal{T}_{1,\ell})\cap D(\mathcal{T}_{p,\ell})\bigr),
\end{equation}
for every $0<2\varepsilon <T$.

(b) We choose $0<3\varepsilon < T$, 
$0<|h|<\varepsilon$, $3\varepsilon\le t\le T-\varepsilon$, write 
\[
\tfrac{u(t+h)-u(t)}{h} = \tfrac{e^{|h|\mathcal{L}_{p,\ell}}-I}{h} u\bigl(t+(h\wedge 0)\bigr) 
+ \frac{1}{h}\int_0^h e^{(|h|-\tau)\mathcal{L}_{p,\ell}} f\bigl(t+(h\wedge 0)+\tau\bigr)d\tau,
\]
and send $h\to 0$. In view of \eqref{eq4.17}, the right-hands side of 
the last identity converges to $\mathcal{L}_{p,\ell} u + \mathcal{C}(u, u)
\in C\bigl([3\varepsilon, T-\varepsilon], Y^p_\ell\bigr)$.
Consequently, the limit of the left-hand side also exists, i.e. $u$ is differentiable and 
$u_t \in C\bigl([3\varepsilon, T-\varepsilon], Y^p_\ell\bigr)$. 
We conclude that the identity
\[
u_t = \mathcal{L}_{p,\ell} u + \mathcal{C}(u, u), \quad 3\varepsilon < t \le T-\varepsilon,
\]
holds in $Y^p_\ell$ and that the mild solutions of Theorem~\ref{lm2.2}(i) are indeed classical.

(c) It remains to verify the continuous dependence of solutions on input data. Let 
\[
u_1, u_2 \in C\bigl([0,T], Y^p_\ell\bigr) \cap C^1\bigl((0,T), Y^{p}_\ell\bigr) 
\cap C\bigl((0,T], D(\mathcal{T}_{1,\ell}) \cap D(\mathcal{T}_{p,\ell})\bigr),
\] 
be two classical solutions of \eqref{eq1.1} associated with data $u_{1,0}, u_{2,0}\in X^p_\ell$ and let 
$e_0 = u_{1,0}-u_{2,0}$, $e = u_1-u_2$. Since both solutions satisfy 
\eqref{eq4.9}, we have
\[
e(t) = e^{t\mathcal{L}_{p,\ell}} e_0 + 
\int_{0}^t e^{(t-\tau)\mathcal{L}_{p,\ell}} \mathcal{C}\bigl(e(\tau), u_1(\tau)+u_2(\tau)\bigr)d\tau
\] 
and then
\begin{align}
\label{eq4.18}
&v(t) \le c_1 \|e_0\|_{Y^p_{\ell}} t^{-\rho} + 
c_2 \int_0^t (t-\tau)^{-\rho-\frac{d}{2p\delta}} v(\tau) d\tau,\\
\nonumber
&v(t):= \bigl\|e(t)\bigr\|_{Y^p_{\ell}}t^{-\rho} + \bigl\|e(t)\bigr\|_{Y^p_{\ell,\rho}},
\end{align}
where $c_1,c_2\ge 1$ are controlled by $p$, $\ell$, $\rho$, $\kappa(\alpha)$, 
$\kappa_\delta(\alpha,\beta)$, $c_\varkappa$ and, in addition, $c_2\ge 1$ depends on 
$\|u_i\|_{C([0,T], Y^p_\ell)}$ and $\|u_i\|_{C_\rho ((0,T], Y^p_{\ell,\rho})}$.
Lemmas~\ref{lm4.3}, \ref{lm4.6} and \eqref{eq4.18} allow us to conclude that
\[
\|e\|_{C([0,T], Y^p_\ell)}+\|e\|_{C_\rho((0,T], Y^p_{\ell,\rho})} \le c\|e_0\|_{Y^p_\ell},
\]
for some $c\ge 1$. The last inequality settles Theorem~\ref{lm2.2}(ii).
 \end{proof}

\begin{remark}\label{lm4.11}
For smooth data $u_0\in D(\mathcal{T}_{1,\ell})\cap D(\mathcal{T}_{p,\ell})$, the mild solutions are strict.
Indeed, using Remarks~\ref{lm4.7}, \ref{lm4.10} and repeating the proof 
of Theorem~\ref{lm2.2}(ii) almost verbatim, we arrive at 
\begin{align}\label{eq4.19}
u\in C\bigl([0,T], Y^p_\ell\bigr) \cap C^1\bigl([0,T], Y^{p}_\ell\bigr) 
\cap C\bigl([0,T], D(\mathcal{T}_{1,\ell})\cap D(\mathcal{T}_{p,\ell})\bigr),
\end{align}
with the same $0<T'\le T$ as in Remarks~\ref{lm4.7} and \ref{lm4.10}.
\end{remark}

\subsection{Positivity and mass conservation}\label{sec4.4}

Finally, Theorem~\ref{lm2.2}(i)-(ii) and Remarks~\ref{lm4.7}, \ref{lm4.10} and \ref{lm4.11}
allow us to settle the last claim of Theorem~\ref{lm2.2}.

\begin{proof}[Proof of Theorem~\ref{lm2.2}(iii)]
(a) To begin, we assume that initial data is very regular, i.e. 
$u_0\in Y^p_{\ell,+}\cap D(\mathcal{T}_{1,\ell})\cap D(\mathcal{T}_{p,\ell})$, and consider the auxiliary problem
\begin{equation}\label{eq4.20}
\bar{u}_t = \mathcal{L}_{p,\ell} \bar{u} + \mathcal{C}\bigl(|\bar{u}|, |\bar{u}|\bigr),\quad t>0,\quad u(0) = u_0.
\end{equation}
Since $\mathcal{C}\bigl(|\cdot|,|\cdot|\bigr)$ satisfies conclusions of Lemma~\ref{lm4.4}, 
Remarks~\ref{lm4.7}, \ref{lm4.10} and \ref{lm4.11} apply and we conclude that the solution 
to \eqref{eq4.20} is strict, i.e. 
\begin{equation}\label{eq4.21}
\bar{u}\in C\bigl([0,T], Y^p_\ell\bigr) \cap C^1\bigl([0,T], Y^{p}_\ell\bigr) \cap 
C\bigl([0,T], D(\mathcal{T}_{1,\ell})\cap D(\mathcal{T}_{p,\ell})\bigr).
\end{equation}
We let 
\begin{align*}
&\bar{v} := c_\varkappa C_\beta \|\bar{u}\|_{L^1(\mathbb{R}_+, w_\ell \underline{\beta}^\rho d\xi)},\\
&\mathcal{K}(u) := - \mathcal{E}_{\bar{v},p,\ell}u + \mathcal{C}\bigl(|u|,|u|\bigr).
\end{align*}
By virtue of Remark~\ref{lm4.10},
\[
\bar{v}\in C^\mu \bigl([0,T], L^{2p}_+(\mathbb{R}^d)\bigr),\quad 0\le \mu<1-\rho
\]
and therefore, the family 
$\bigl\{\bigl(\mathcal{L}_{p,\ell}+\mathcal{E}_{\bar{v},p,\ell}(t), D(\mathcal{T}_{p,\ell})\bigr)\bigr\}_{t\in[0,T]}$ 
fells in the scope of Lemma~\ref{lm4.5}. 

We employ the last fact, bound \eqref{eq4.6} and repeat calculations of Remarks~\ref{lm4.7}, \ref{lm4.10} 
and \ref{lm4.11} one more time, to conclude that the second auxiliary problem
\begin{subequations}\label{eq4.22}
\begin{equation}\label{eq4.22a}
\hat{u}_t = \bigl[\mathcal{L}_{p,\ell} + \mathcal{E}_{\bar{v},p}\bigr]\hat{u} 
+ \mathcal{K}(\hat{u}),\quad t>0,\quad u(0) = u_0.
\end{equation}
is also strictly solvable in $X^p_{\ell}$ and that its strict solution 
\begin{equation}\label{eq4.22b}
\hat{u}\in C\bigl([0,T], Y^p_\ell\bigr) \cap C^1\bigl([0,T], Y^{p}_\ell\bigr) \cap 
C\bigl([0,T], D(\mathcal{T}_{1,\ell})\cap D(\mathcal{T}_{p,\ell})\bigr),
\end{equation}
is given by the variation of constant formula
\begin{equation}\label{eq4.22c}
\hat{u}(t) = \mathcal{U}_{\bar{v}, p,\ell}(t,0) u_0 
+ \int_{0}^t \mathcal{U}_{\bar{v}, p,\ell}(t,\tau) \mathcal{K}\bigl(\hat{u}(\tau)\bigr)d\tau,\quad 0\le t\le T.
\end{equation}
\end{subequations}

In view of \eqref{eq4.22b}, formula \eqref{eq4.22a} implies that $\hat{u}$ solves \eqref{eq4.20} 
and (by the uniqueness of strict solutions) that $\hat{u} = \bar{u}$. This fact, together with our 
definition of $\bar{v}$ implies that $\bigr\{\mathcal{K}\bigl(\hat{u}(t)\bigr)\bigr\}_{0\le t\le T} 
\subset Y^p_{\ell,+}$. 
Since the evolution operator $\mathcal{U}_{\bar{v},p,\ell}$ is also positive, it follows from 
\eqref{eq4.22c} that $\bar{u} = \hat{u} \in C\bigl([0,T], Y^p_{\ell,+}\bigr)$. Finally, positivity of $\hat{u}$ implies
$\mathcal{C}\bigl(|\hat{u}|, |\hat{u}|\bigr) = \mathcal{C}(\hat{u}, \hat{u})$ and therefore 
$\bar{u} = \hat{u}$ solves \eqref{eq1.1}. We conclude that for regular positive data 
strict solutions of \eqref{eq1.1} are indeed positive.

(b) For general data $u_0\in Y^p_{\ell,+}$ and $h>0$, we have  
\[
u_{0,h} = \tfrac{1}{h}\int_0^h e^{\tau \mathcal{L}_{p,\ell}} u_0 d\tau \in Y^p_{\ell,+}\cap 
D(\mathcal{T}_{1,\ell})\cap D(\mathcal{T}_{p,\ell}).
\]
Since $u_{0,h}\to u_0$ as $h\to0^+$ in $Y^{p}_{\ell,+}$ and since 
$Y^p_{\ell,+}$ and $C\bigl([0,T], Y^p_{\ell,+}\bigr)$ are closed, the first part of Theorem~\ref{lm2.2}(iii) 
follows from the continuous dependence of classical solutions on the input data.

(c) By our assumptions, $\ell> 1$ and $u_0\in Y^p_{\ell,+}$. Since in this case 
the classical solution $u$ is positive, we fix some $0<t\le T$, multiply \eqref{eq1.1} by $\xi$ and 
integrate both sides of the equation over $\mathbb{R}^d\times\mathbb{R}_+$. 
Thanks to the special structure of the diffusion, fragmentation and coagulation operators, 
direct application of the divergence and the Fubini theorems gives
\[
\tfrac{d}{dt}\|u\|_{X^1_\xi} = 0, \quad 0<t\le T.
\]
Upon integration, $\bigl\|u(\tau)\bigr\|_{X^1_\xi} = \bigl\|u(t)\bigr\|_{X^1_\xi}$, 
$0<\tau\le t\le T$. Since, $u\in C\bigl([0,T], Y^p_\ell\bigr)$, it follows that 
$\bigl\|u(t)\bigr\|_{X^1_\xi} = \lim_{\tau\to0^+} \bigl\|u(\tau)\bigr\|_{X^1_\xi} = \|u_0\|_{X^1_\xi}$ 
and the second claim of Theorem~\ref{lm2.2}(iii) is also settled.
\end{proof}

To conclude this section, we remark that the artificial restriction $T\le 1$, 
that appears in the proof of Theorem~\ref{lm2.2}(i), is imposed for the sake 
of convenience only.  The mild (respectively classical) solution emanating from $u_0\in Y^{p}_{\ell}$,
extends to its maximal interval of existence $\bigl[0, T(u_0)\bigr)$ and 
the standard continuation alternative holds:
\begin{itemize}
\item[(i)]
either $T(u_0)=\infty$ and the solution is global; 
\item[(ii)] or $T(u_0)<\infty$ and $\bigl\|u(t)\bigr\|_{Y^p_{\ell}} \to\infty$ as $t\to T(u_0)^-$. 
\end{itemize}
\section{Global solutions}\label{sec5}

In this Section, we restrict our attention to \eqref{eq1.1}, equipped with power rates that satisfy \eqref{eq2.9}. 
In this special case, Theorem~\ref{lm2.2} definitely applies and positive classical solutions $u(t)$ associated to 
data $u_0\in Y^p_{\ell,+}$, $\ell>\bar{\ell}_0\vee \ell_{2p}$, are, at least locally, well defined. In what follows, 
we show that under \eqref{eq2.4}, the $Y^p_{\ell}$-norms of these solutions remain bounded 
in bounded time intervals. 
\begin{lemma}\label{lm5.2}
Assume $0< \theta_\beta \le 1 - \bar{\ell}_0$. Then, in the settings of Theorem~\ref{lm2.3}, 
the classical positive solutions to \eqref{eq1.1} satisfy
\begin{equation}\label{eq5.3}
\bigl\|u(t)\bigr\|_{X^1_1} \le e^{\omega t} \|u_0\|_{X^1_1},\quad 0\le t<T,
\end{equation}
with some $\omega\ge 1$.
\end{lemma}
\begin{proof}
Thanks to the positivity of $u$ and special structure of the diffusion, the fragmentation 
and the coagulation operators, direct application of the divergence and the Fubini theorem gives
\[
\tfrac{d}{dt} \bigl\|u(t)\bigr\|_{X^1} \le c' \bigl\|u(t)\bigr\|_{X^1_{\theta_\beta+\ell_0}} 
\le c \bigl\|u(t)\bigr\|_{X^1_{1}},
\]
with some $c',c\ge 1$. Combining this bound with Theorem~\ref{lm2.2}(iii) and then using Gronwall's 
inequality, we arrive at \eqref{eq5.3}.
\end{proof}

To proceed further, we use $\bigl(\underline{\mathcal{T}}_{p,\ell}, D(\underline{\mathcal{T}}_{p,\ell})\bigr)$ 
to denote the $X^p_\ell$, $1\le p<\infty$, realization of  
\[
\underline{\mathcal{T}}_0 [u](x,\xi) := \nabla^T \alpha(\xi,x)\nabla u(x,\xi) - \underline{\beta}(\xi) u(x,\xi),
\]
from Theorem~\ref{lm3.3} and let $\bigl(\Delta_p, D(\Delta_p)\bigr)$ be the 
$L^p(\mathbb{R}^d)$ realization of the standard $d$-dimensional Laplacian  
as described in Section~\ref{sec3.1}. On the account of \eqref{eq2.9a}, for a.e. $\xi\in\mathbb{R}_+$,
actions of $\{e^{t\underline{\mathcal{T}}_{p,\ell}}\}_{t\ge 0}$, $1\le p<\infty$, are given explicitly by 
\[
\bigl[e^{t\underline{\mathcal{T}}_{p,\ell}} u\bigr](x,\xi) 
= \bigl[e^{-t\underline{\beta}(\xi)} e^{t\alpha(\xi)\Delta_p} u\bigr](x, \xi),
\quad t>0.
\] 
Since, 
 \[
t_0^{\frac{d}{2}} e^{t_0 \Delta_p} u \le t_1^{\frac{d}{2}} e^{t_1 \Delta_p} u, \quad 
0<t_0\le t_1,\quad u\in L^p_+(\mathbb{R}^d),\quad 1\le p\le\infty,
\]
it follows that
\begin{equation}\label{eq5.4}
\alpha(\xi_1)^{\frac{d}{2}}\bigl[e^{t\underline{\mathcal{T}}_{p,\ell}(\xi_1)} u\bigr](x,\xi_1) \le  
\alpha(\xi_0)^{\frac{d}{2}}\bigl[e^{t\underline{\mathcal{T}}_{p,\ell}(\xi_0)} u\bigr](x,\xi_0),\; 
\text{for a.e. $0<\xi_0\le \xi_1$},
\end{equation}
for $u\in X^p_{\ell,+}$, $1\le p<\infty$, and $t>0$.

The \emph{size-monotonicity property} \eqref{eq5.4} yields
\begin{lemma}\label{lm5.3}
In the settings of Lemma~\ref{lm5.2}, the positive classical solutions satisfy
\begin{subequations}\label{eq5.5}
\begin{equation}\label{eq5.5a}
\bigl\|u(t)\bigr\|_{X^{p_1}_r} \le c_r\|e^{t\underline{\mathcal{T}}_{p_1,\ell}}u_0\|_{X^{p_1}_r}
+ c_r\int_0^t (t-\tau)^{\theta - \frac{d}{2\delta}(\frac{1}{p_0}-\frac{1}{p_1})-1} 
\bigl\|u(\tau)\bigr\|_{X^{p_0}_{r+\theta_\beta\theta}} d\tau,
\end{equation}
$0\le t<T$, with some $c_r\ge 1$, provided
\begin{align}
\label{eq5.5b}
&0\le \tfrac{d}{2\delta}\bigl(\tfrac{1}{p_0} - \tfrac{1}{p_1}\bigr) < \theta <1, \quad 1\le p_0\le p_1 \le p<\infty,\\
\label{eq5.5c}
&\bar{\ell}_1 + (1-\theta)\theta_\beta < r \le (1- \theta_\alpha)\wedge (\ell - \theta_\beta\theta).
\end{align}
\end{subequations}
\end{lemma}
\begin{proof}
(a) In the settings of Lemma~\ref{lm5.3}, we have 
$Y^p_\ell\hookrightarrow X^{p_1}_r\cap X^{p_0}_{r+\theta_\beta\theta}$, therefore
\begin{align*}
u(t) &= e^{t\underline{\mathcal{T}}_{p_1,\ell}} u_0 
 - \int_0^t e^{(t-\tau)\underline{\mathcal{T}}_{p_1,\ell}} [(\beta - \underline{\beta}) u(\tau)] d\tau\\
&\qquad\qquad\quad
+ \int_0^t e^{(t-\tau)\underline{\mathcal{T}}_{p_1,\ell}} \mathcal{B}_{p_0,\ell}^+[u(\tau)] d\tau
+ \int_0^t e^{(t-\tau)\underline{\mathcal{T}}_{p_1,\ell}} \mathcal{C}\bigl(u(\tau),u(\tau)\bigr) d\tau\\
&\le e^{t\underline{\mathcal{T}}_{p_1,\ell}} u_0 
+ \int_0^t e^{(t-\tau)\underline{\mathcal{T}}_{p_1,\ell}} \mathcal{B}_{p_0,\ell}^+[u(\tau)] d\tau
+ \int_0^t e^{(t-\tau)\underline{\mathcal{T}}_{p_1,\ell}} \mathcal{C}\bigl(u(\tau),u(\tau)\bigr) d\tau\\
& =: e^{t\underline{\mathcal{T}}_{p_1,\ell}} u_0 + J_0(t) + J_1(t),\quad 0\le t<T.
\end{align*}
We let $w(\xi) := w_{r+\theta_\alpha}(\xi)\alpha(\xi)^{\frac{d}{2}}$ and remark that 
$w(\xi)\approx w_r(\xi)$ in $\mathbb{R}_+$ and that 
\begin{equation}\label{eq5.6}
w_{r+\theta_\alpha}(\xi+\eta) \le w_{r+\theta_\alpha}(\xi)+w_{r+\theta_\alpha}(\eta),\quad 
\xi,\eta\in\mathbb{R}_+,
\end{equation}
thanks to \eqref{eq2.9a} and the second inequality of \eqref{eq5.5c}, respectively. 
As $J_1(t)\in X^{p_1}_r$, it follows that 
$J_1(t)(x)\in L^1(\mathbb{R}_+,wd\xi)$ for a.e. $x\in \mathbb{R}^d$. Hence, we may use the Fubini 
theorem, the size-monotonicity of the loss diffusion semigroup \eqref{eq5.4} and the elementary 
inequality \eqref{eq5.6}, to conclude that
\begin{align*}
\qquad
\int_{\mathbb{R}_+}  J_1(t) w(\xi)d\xi
&= \tfrac{1}{2}\int_{0}^t \int_{\mathbb{R}_+} \int_{\mathbb{R}_+}  
\Bigl(w(\xi+\eta) e^{(t-\tau)\underline{\mathcal{T}}_{p,\ell}(\xi+\eta)} \\
&\qquad\qquad\qquad\quad
- w(\xi) e^{(t-\tau)\underline{\mathcal{T}}_{p,\ell}(\xi)} 
- w(\eta) e^{(t-\tau)\underline{\mathcal{T}}_{p,\ell}(\eta)}\Bigr) \\ 
&\qquad\qquad\qquad\quad
[\varkappa(\cdot,\xi,\eta)u(\cdot,\xi)u(\cdot,\eta)]
d\xi d\eta d\tau \le 0,
\end{align*}
a.e. in $\mathbb{R}^d$. Since $u$ is positive, this gives
\begin{align*}
\bigl\|u(t)\bigr\|_{L^1(\mathbb{R}_+, w_rd\xi)} \le  
c\|e^{t\underline{\mathcal{T}}_{p_1,\ell}} u_0\|_{L^1(\mathbb{R}_+, w_rd\xi)}
+ c\bigl\|J_0(t)\bigr\|_{L^1(\mathbb{R}_+, w_rd\xi)},\quad 0\le t<T,
\end{align*}
a.e. in $\mathbb{R}^d$ and then
\[
\bigl\|u(t)\bigr\|_{X^{p_1}_r} \le c\|e^{t\underline{\mathcal{T}}_{p_1,\ell}}u_0\|_{X^{p_1}_r} 
+ c\bigl\|J_0(t)\bigr\|_{X^{p_1}_r},\quad 0\le t<T,
\]
where $c\ge 1$ depends on the equivalence constants from \eqref{eq2.9a}.

(b) To bound $J_0(t)$, we note that
\[
\mathcal{B}_{p_0,\ell}^+\in \mathcal{L}\bigl(X^{p}_{\ell,1}, X^{p}_{\ell}\bigr),\quad 1\le p<\infty,
\quad \bar{\ell}_1< \ell,
\]
while under assumption \eqref{eq2.9b}, we have
\[
\|\cdot\|_{X^p_{\ell,s}} \approx \|\cdot\|_{X^p_{\ell+s\theta_\beta}}, \quad 1\le p<\infty,\quad 
0\le \ell,\ell+s\theta_\beta.
\]
Hence, \eqref{eq3.7c}, \eqref{eq2.4} and the first inequality in \eqref{eq5.5c} guarantee that
\begin{align*}
\bigl\|J_0(t)\bigr\|_{X^{p_1}_r} 
&\le c\int_0^t \bigl\|e^{(t-\tau)\underline{\mathcal{T}}_{p_1,\ell}} 
\mathcal{B}_{p_0,\ell}^+[u(\tau)]\bigr\|_{X^{p_1}_r} d\tau\\
&\le c\int_0^t (t-\tau)^{\theta-1 - \frac{d}{2\delta}(\frac{1}{p_0}-\frac{1}{p_1})} 
\bigl\|\mathcal{B}_{p_0,\ell}^+[u(\tau)]\bigr\|_{X^{p_0}_{r-(1-\theta)\theta_\beta}}d\tau\\
&\le c\int_0^t (t-\tau)^{\theta-1 - \frac{d}{2\delta}(\frac{1}{p_0}-\frac{1}{p_1})} 
\bigl\|u(\tau)\bigr\|_{X^{p_0}_{r+\theta\theta_\beta}}d\tau,\quad 0\le t<T.
\end{align*}
The last bound completes the proof.
\end{proof}

As an immediate consequence of Lemmas~\ref{lm5.2} and \ref{lm5.3}, we have
\begin{corollary}\label{lm5.4}
Assume that conditions of Lemma~\ref{lm5.2} are satisfied and that 
\begin{subequations}\label{eq5.7}
\begin{equation}\label{eq5.7a}
0<\theta_\alpha<\theta_\beta < \tfrac{2\delta(1-\bar{\ell}_1)p'}{d+2\delta p'} \wedge (1-\bar{\ell}_0),
\end{equation}
then
\begin{equation}\label{eq5.7b}
\bigl\|u(t)\bigr\|_{X^p_{\bar{\ell}_1 + \theta_\beta - \theta_\alpha}} \le c e^{w t} \|u_0\|_{Y^p_\ell},\quad 0\le t< T,
\end{equation}
with some $c,w\ge 1$ that depend on $p$, $\ell$ and the coefficients of \eqref{eq1.1} only.
\end{subequations}
\end{corollary}
\begin{proof}
(a) To begin we let 
\[
p_0 :=1,\quad r_0 :=1,\quad v_0(t) := \bigl\|u(t)\bigr\|_{X^{p_0}_{r_0}},\quad 0\le t<T.
\]
By Lemma~\ref{lm5.2}, $v(t) \le e^{\omega t} \|u_0\|_{X^1_1}$. 
We fix some $0<\varepsilon<1$ 
and assume for a moment that $0<\theta< 1$ can be chosen so that 
\[
\bar{\ell}_1 + (1-\theta)\theta_\beta < r_0 - \theta_\beta\theta 
\le (1-\theta_\alpha)\wedge(\ell-\theta_\beta\theta). 
\]
Then letting,
\begin{align*}
v_1(t) := \bigl\|u(t)\bigr\|_{X^{p_0}_{r_0}},\quad r_1 :=  r_0 - \theta_\beta\theta,\quad 
\tfrac{1}{p_1} := \tfrac{1}{p} \vee \bigl(\tfrac{1}{p_0} - \tfrac{2\delta}{d}\theta\varepsilon \bigr)
> \tfrac{1}{p_0} - \tfrac{2\delta}{d}\theta.
\end{align*}
and using Lemma~\ref{lm5.3}, we arrive at
\begin{align*}
v_1(t) &\le c_{r_1}\|e^{t\underline{\mathcal{T}}_{p_1,\ell}}u_0\|_{X^{p_1}_{r_1}}
+ c_{r_1}\int_0^t (t-\tau)^{\theta-1 - \frac{d}{2\delta}(\frac{1}{p_0}-\frac{1}{p_1})} 
v_0(\tau) d\tau\\
&\le c_{r_1}\|u_0\|_{X^{p_1}_{r_1}} +  c_{r_1}\int_0^t 
\tau^{\theta-1 - \frac{d}{2\delta}(\frac{1}{p_0}-\frac{1}{p_1})}  d\tau
e^{\omega t} \|u_0\|_{X^{p_0}_{r_0}} \\
&\le c_1 e^{\omega_1 t} \|u_0\|_{X^{p}_{\ell}},\quad 0\le t<T,
\end{align*}
with some $c_1,\omega_1>0$. 

(b) We may repeat step (a) iteratively to obtain
\begin{subequations}\label{eq5.8}
\begin{align}
\label{eq5.8a}
&r_{k+1} :=  r_k - \theta_\beta\theta,\quad 
\tfrac{1}{p_{k+1}} := \tfrac{1}{p} \vee \bigl(\tfrac{1}{p_k} - \tfrac{2\delta}{d}\theta\varepsilon \bigr)
> \tfrac{1}{p_k} - \tfrac{2\delta}{d}\theta,\quad 1\le k\le n,\\
\label{eq5.8b}
&v_{k+1}(t) := \bigl\|u(t)\bigr\|_{X^{p_{k+1}}_{r_{k+1}}} \le c_{k+1} e^{\omega_{k+1} t} \|u_0\|_{X^{p}_{\ell}},
\quad 0\le t<T.
\end{align}
The calculations make sense as long as 
\begin{equation}\label{eq5.8c}
\bar{\ell}_1 + (1-\theta)\theta_\beta < r_{n+1} < r_1\le (1-\theta_\alpha)\wedge(\ell-\theta_\beta\theta). 
\end{equation}
\end{subequations}
If $n\ge 0$ is the smallest integer that satisfies $p_{n+1}= p$, then \eqref{eq5.8} holds 
for some $0<\varepsilon<1$, provided
\begin{equation}\label{eq5.9}
\tfrac{d}{2\delta p'\theta} - 1 < n < \tfrac{1-\bar{\ell}_1-\theta_\beta}{\theta_\beta \theta},
\quad \theta_\alpha\le \theta\theta_\beta.
\end{equation}
We let $\theta := \tfrac{\theta_\alpha}{\theta_{\beta}}$, then in view of the first inequality 
in \eqref{eq5.7a}, $0<\theta<1$. We combine this bound with the second inequality 
of \eqref{eq5.7a} to conclude that indeed, \eqref{eq5.9} holds for some non-negative 
integer $n\ge 0$. The proof is complete. 
\end{proof}

Now we are in the position to prove Theorem~\ref{lm2.3}.
\begin{proof}[Proof of Theorem~\ref{lm2.3}]
By the standard continuation alternative, it suffices to show that 
$\sup_{0\le t<T} \bigl\|u(t)\bigr\|_{Y^p_\ell} < \infty$, for any given $0< T<\infty$. 

(a)  In our settings, 
$0<\rho\theta_\beta \le \bar{\ell}_1+\theta_\beta - \theta_\alpha$. Therefore, 
we may use \eqref{eq4.9}, \eqref{eq2.7}, 
Lemmas~\ref{lm4.2}-\ref{lm4.4}, \eqref{eq2.9b} and \eqref{eq5.7b}, to obtain
\begin{align*}
\bigl\|u(t)\bigr\|_{X^p_{\ell+j\rho\theta_\beta}} 
&\le c_1 t^{-j\rho}\|u_0\|_{Y^p_\ell}+ c_2' \int_0^t (t-\tau)^{-j\rho - \frac{d}{2p\delta}}
\bigl[\bigl\|u(\tau)\bigr\|_{X^p_{\ell+\rho\theta_\beta}} \bigl\|u(\tau)\bigr\|_{X^p_0}\\
&\qquad\qquad\qquad\qquad\qquad\qquad\qquad\qquad
+\bigl\|u(\tau)\bigr\|_{X^p_{\ell}} \bigl\|u(\tau)\bigr\|_{X^p_{\rho\theta_\beta}}\bigr]d\tau\\
&\le c_1 t^{-j\rho}\|u_0\|_{Y^p_\ell} + c_2\|u_0\|_{Y^p_\ell}
\int_0^t (t-\tau)^{-j\rho - \frac{d}{2p\delta}} \bigl\|u(\tau)\bigr\|_{X^p_{\ell+\rho\theta_\beta}} d\tau,
\end{align*}
for $j=0,1$, some $c_1,c_2\ge1$ and any $0<t\le T$. 
We apply Lemma~\ref{lm4.6} to the last inequality with $j=1$ and then substitute the result back 
into the same inequality but with $j=0$. This gives
\begin{equation}\label{eq5.10}
t^{j\rho}\bigl\|u(t)\bigr\|_{X^p_{\ell+j\rho\theta_\beta}} \le C_1 \|u_0\|_{Y^p_\ell},
\quad j=0,1,\quad 0<t\le T,
\end{equation} 
with some $1\le C_1<\infty$ that depends on $\|u_0\|_{Y^p_\ell}$, $T>0$ and coefficients 
of the model \eqref{eq1.1} only.

(b) The same arguments as above, give
\begin{align*}
\bigl\|u(t)\bigr\|_{X^1_{\ell+j\rho\theta_\beta}} 
&\le c_3' t^{-j\rho}\|u_0\|_{Y^1_\ell}+ c_4' \int_0^t (t-\tau)^{-j\rho}
\bigl[\bigl\|u(\tau)\bigr\|_{X^2_{\ell+\rho\theta_\beta}} \bigl\|u(\tau)\bigr\|_{X^2_0}\\
&\qquad\qquad\qquad\qquad\qquad\qquad\qquad
+\bigl\|u(\tau)\bigr\|_{X^2_{\ell}} \bigl\|u(\tau)\bigr\|_{X^2_{\rho\theta_\beta}}\bigr]d\tau.
\end{align*}
for $j=0,1$, some $c_3',c_4'\ge1$ and any $0<t\le T$. 
As in the proof of Theorem~\ref{lm2.2}(i), we observe that
\[
\|u\|_{X^2_{r}} \le 
\|u\|_{X^1_{r}}^{1-\frac{p'}{2}}\|u\|_{X^p_{r}}^{\frac{p'}{2}},
\quad 2\le p<\infty,\quad r\ge 0.
\] 
Hence, using Lemma~\ref{lm5.2}, \eqref{eq5.10} and Young's inequality, we infer
\begin{align*}
\bigl\|u(\tau)\bigr\|_{X^2_{\ell+\rho\theta_\beta}} \bigl\|u(\tau)\bigr\|_{X^2_0}
&\le 
\bigl\|u(\tau)\bigr\|_{X^1_{\ell+\rho\theta_\beta}}^{1-\frac{p'}{2}}
\bigl\|u(\tau)\bigr\|_{X^p_{\ell+\rho\theta_\beta}}^{\frac{p'}{2}}
\bigl\|u(\tau)\bigr\|_{X^1_0}^{1-\frac{p'}{2}}
\bigl\|u(\tau)\bigr\|_{X^p_0}^{\frac{p'}{2}}\\
&\le \tau^{-\frac{p'}{2}\rho} \bigl\|u(\tau)\bigr\|_{X^1_{\ell+\rho\theta_\beta}}^{1-\frac{p'}{2}} 
\bigl(C_1 \|u_0\|_{Y^p_\ell}\bigr)^{1+\frac{p'}{2}}\\
&\le  \bigl(1-\tfrac{p'}{2}\bigr)\bigl\|u(\tau)\bigr\|_{X^1_{\ell+\rho\theta_\beta}}
+\tfrac{p'}{2} \tau^{-\rho} \bigl(C_1 \|u_0\|_{Y^p_\ell}\bigr)^{1+\frac{2}{p'}}.
\end{align*}
Since similar calculations give also
\[
\bigl\|u(\tau)\bigr\|_{X^2_\ell} \bigl\|u(\tau)\bigr\|_{X^2_{\rho\theta_\beta}}
\le  \bigl(1-\tfrac{p'}{2}\bigr)\bigl\|u(\tau)\bigr\|_{X^1_{\ell+\rho\theta_\beta}}
+\tfrac{p'}{2} \bigl(C_1 \|u_0\|_{Y^p_\ell}\bigr)^{1+\frac{2}{p'}},
\]
after some elementary manipulations, we arrive at
\begin{align*}
\bigl\|u(t)\bigr\|_{X^1_{\ell+j\rho\theta_\beta}} 
&\le c_3 t^{-j\rho}\|u_0\|_{Y^p_\ell} + c_4 \|u_0\|_{Y^p_\ell}^{1+\frac{2}{p'}}
+ c_5\int_0^t (t-\tau)^{-j\rho} \bigl\|u(\tau)\bigr\|_{X^1_{\ell+\rho\theta_\beta}} d\tau,
\end{align*}
for $j=0,1$, some $c_3,c_4,c_5\ge1$ and $0<t\le T$. 
As in part (a) of the proof, we apply Lemma~\ref{lm4.6} to obtain
\begin{equation}\label{eq5.11}
t^{j\rho}\bigl\|u(t)\bigr\|_{X^1_{\ell+j\rho\theta_\beta}} \le C_2 \bigl(\|u_0\|_{Y^p_\ell}
+\|u_0\|_{Y^p_\ell}^{1+\frac{2}{p'}}\bigr),\quad j=0,1,\quad 0<t\le T,
\end{equation} 
where $1\le C_2< \infty$ depends on $\|u_0\|_{Y^p_\ell}$, $T>0$ and coefficients 
of the model \eqref{eq1.1} only. Bounds \eqref{eq5.10}-\eqref{eq5.11} guarantee that
$\bigl\|u(t)\bigr\|_{Y^p_\ell}$ cannot blow up in any given bounded time interval $[0,T)$. 
Hence, $u(t)$ is defined globally and the proof is complete.
\end{proof}

\section{Conclusion}\label{sec6}

In the paper, we employed the theory of positive $C_0$-semigroups and semi-linear equations
to study local and global classical solvability of the continuous diffusion-fragmenta\-tion-coagulation 
models with vanishing diffusion and unbounded fragmentation and coagulation rates. 
In our functional settings, the well developed theory of UMD-valued linear parabolic problems 
\cite{DenHiePru2003, KunWei2004} as well as vector-valued generation results of \cite{Am1997} 
do not apply. In Section~\ref{sec3}, we adopted an alternative semigroup generation technique 
that uses a combination of duality, sharp Green's functions estimates available for second-order elliptic 
operators in the divergence form and perturbation results. 
The vector-valued $C_0$-semigroup, obtained this way, behaves as a fractional scalar diffusion 
in space and has moment-regularizing property in particles sizes. These facts were employed 
subsequently in Section~\ref{sec4} to obtain local classical solutions to the complete nonlinear 
model \eqref{eq1.1}. Existence of global in time positive classical solutions requires additional 
structural restrictions on the model coefficients. In Section~\ref{sec5}, we were able to construct 
such solutions in the special case of power diffusion, fragmentation and coagulation rates. 
To the best of our knowledge this is the first result of this type in the literature. 

To conclude this paper, we note that all arguments and results of Sections~\ref{sec3} and \ref{sec4} 
extend verbatim to problems with vanishing diffusion on smooth domains $\Omega\subset\mathbb{R}^d$, 
subject to the homogeneous Neumann boundary condition. However, extension of our global results 
is not immediate. Calculations of Section~\ref{sec5} rely on the size-monotonicity property \eqref{eq5.4} 
that is neither obvious nor immediately available in general domains.

\bibliographystyle{plain}
\bibliography{b2026}
\end{document}